\documentclass[10pt,reqno]{amsart}

\usepackage{amssymb, verbatim, mathrsfs}

\theoremstyle{plain}
\newtheorem{theorem}{Theorem}[section]
\newtheorem{prop}[theorem]{Proposition}
\newtheorem{corollary}[theorem]{Corollary}
\newtheorem{lemma}[theorem]{Lemma}

\theoremstyle{definition}

\newtheorem*{remark*}{Remark}
\newtheorem*{remarks*}{Remarks}

\def\<{\langle}
\def\>{\rangle}
\def\geqs{\geqslant}
\def\leqs{\leqslant}

\def\a{\alpha}
\def\b{\beta}
\def\g{\gamma}
\def\d{\delta}
\def\e{\varepsilon}

\def\s{\sigma}

\def\w{\omega}

\def\z{\zeta}

\def\N{\mathbb N}
\def\V{\mathbb V}
\def\R{\mathbb R}
\def\LL{\mathscr L}
\def\RR{\mathscr R}
\def\/{\kern 0.05em}

\newcommand{\flushpar}{\par\noindent}
\newcommand{\nab}[1]{\nabla\kern-.2em\lower.8ex\hbox{\SMALL$#1$}\/}
\newcommand{\nabsq}[2]{\nabla^2\kern-.3em\lower.8ex\hbox{\SMALL$#1,#2$}\/}
\newcommand{\diffop}[2]{#1\kern-.1em\lower.8ex\hbox{\SMALL$#2$}\/} 
\newcommand{\dotprod}{\/\raise.25ex\hbox{\SMALL$\bullet$}\/}
\newcommand{\uperp}[1]{#1\raise1.1ex\hbox{\SMALL$\bot$}}

\DeclareMathOperator{\trace}{tr} 

\DeclareMathOperator{\im}{im} 
\DeclareMathOperator{\Ric}{Ric} 
\DeclareMathOperator{\Div}{div}

\numberwithin{equation}{section}

\begin{document}
\title{Harmonic Vector Fields on Space Forms}

\author{M. Benyounes}
\address{D\'epartement de Math\'ematiques \\
Universit\'e de Bretagne Occidentale \\
6, avenue Victor Le Gorgeu \\
CS 93837, 29238 Brest Cedex 3\\
France} \email{Michele.Benyounes@univ-brest.fr}

\author{E. Loubeau}
\address{D\'epartement de Math\'ematiques \\
Universit\'e de Bretagne Occidentale \\
6, avenue Victor Le Gorgeu \\
CS 93837, 29238 Brest Cedex 3\\
France} 
\email{Eric.Loubeau@univ-brest.fr}

\author{C.~M. Wood}
\address{Department of Mathematics \\
University of York \\
Heslington, York Y010 5DD\\
U.K.} \email{cmw4@york.ac.uk}



\date{\today}

\begin{abstract}
A vector field $\s$ on a Riemannian manifold $M$ is said to be harmonic if there exists a member of a $2$-parameter family of generalised Cheeger-Gromoll metrics on $TM$ with respect to which $\s$ is a harmonic section.  If $M$ is a simply-connected non-flat space form other than the $2$-sphere, examples are obtained of conformal vector fields that are harmonic.  In particular, the harmonic Killing fields and conformal gradient fields are classified, a loop of non-congruent harmonic conformal fields on the hyperbolic plane constructed, and the $2$-dimensional classification achieved for conformal fields.  A classification is then given of all harmonic quadratic gradient fields on spheres.
\end{abstract}

\maketitle

\section{Introduction}

Let $\s$ be a vector field on an $n$-dimensional Riemannian manifold $(M,g)$.  It was observed in \cite{Nou} and \cite{Ish} that if $M$ is compact and $\s$ is a harmonic map \cite{ES} into $TM$ equipped with the Sasaki metric $h$ \cite{Sas} then $\s$ is parallel; furthermore this remains the case under the less stringent and arguably more natural condition that $\s$ is a harmonic section of $TM$; ie. a critical point of (vertical) energy with respect to variations through vector fields \cite{Wood1}.  This rigidity may be overcome for vector fields of unit length by further restricting the class of variations to unit vector fields; in other words, the study of harmonic sections of the unit tangent bundle.  These {\sl harmonic unit fields\/} have acquired an extensive literature; see for example the bibliography of \cite{Gil}.  Unfortunately its scope is restricted by the topology of $M$.  The essential reason for the rigidity of the energy functional in this context is that the restriction of the Sasaki metric to tangent spaces is flat, creating an analogy between harmonic sections of $TM$ and harmonic (vector-valued) functions on $M$; whereas harmonic sections of the unit tangent bundle are analogous to harmonic maps into spheres \cite{Smi}.  To complete the analogy, ordinary derivatives are replaced by covariant derivatives, and the Laplace-Beltrami operator for functions by the rough Laplacian for vector fields; thus the equations for a harmonic unit field are:
\begin{equation}
\nabla^*\nabla\s=|\nabla \s|^2\s,
\label{harmunit}
\end{equation}
where:
\begin{equation}
\nabla^*\nabla\s=-\trace\nabla^2\s
=-\textstyle\sum_i\nabsq{E_i}{E_i}\s,
\label{roughlap}
\end{equation}
for any local $g$-orthonormal tangent frame $\{E_i\}$ on $M$.  

\par
In \cite{BLW1} it was proposed to address the rigidity problem for vector fields by embedding $h$ in a 2-parameter family $h_{p,q}$ ($p,q\in\R$) of {\sl generalised Cheeger-Gromoll metrics,}  within which $h=h_{0,0}$, the Cheeger-Gromoll metric \cite{CG} appears as $h_{1,1}$, and $h_{2,0}$ is the stereographic metric.  The $h_{p,q}$ all belong to the infinite-dimensional family of $g$-natural metrics on $TM$ \cite{Ab1,Ab2}, but are much more tightly controlled, being constructed from a spherically symmetric family of metrics on $\R^n$ via the Kaluza-Klein procedure.  Then $\s$ is said to be {\sl $(p,q)$-harmonic\/} if $\s$ is a harmonic section of $TM$ with respect to  $h_{p,q}$ (the metric $g$ on $M$ is fixed throughout).  From \cite{BLW1} the corresponding Euler-Lagrange equations are:
\begin{equation}
\tau_{p,q}(\s)=T_p(\s)-\phi_{p,q}(\s)\/\s=0,
\label{harmeqn}
\end{equation}
where:
\begin{equation}
\gathered
T_p(\s)=(1+|\s|^2)\nabla^*\nabla\s+2p\/\nab{\nabla F}\s, \\
\phi_{p,q}(\s)=p\/|\nabla\s|^2-pq\/|\nabla F|^2-q(1+|\s|^2)\Delta F.
\endgathered
\label{harmterms}
\end{equation}
\medskip
\flushpar
Here $F$ is an abbreviation for $\tfrac12\/|\s|^2$, $\nabla F$ denotes the gradient vector field, and $\Delta F$ the Laplacian.
If $\s$ is a unit field then it is easily seen that \eqref{harmeqn} reduces to \eqref{harmunit} with $p=2$; thus harmonic unit fields are $(2,q)$-harmonic, for all $q$.  If $\s$ is parallel then $\s$ is $(p,q)$-harmonic for all $(p,q)$.

\par
A notable feature of the generalised Cheeger-Gromoll family is that if $q<0$ then $h_{p,q}$ has variable signature.  Precisely, if for $r>0$ the ball and sphere subbundles of $TM$ are denoted:
$$
BM(r)=\{X\in TM:|X|<r\},
\qquad
SM(r)=\{X\in TM:|X|=r\},
$$
and $q<0$, then $h_{p,q}$ is Riemannian on the {\sl $q$-Riemannian ball bundle\/} $BM(r)$ where $r=\sqrt{-1/q}$, degenerate (in radial vertical directions) on $SM(r)$, and Lorentz\-ian on the exterior of $BM(r)$.  However, for all $r>\sqrt{-1/q}$ the sphere bundle $SM(r)$ is a spacelike hypersurface of $(TM,h_{p,q})$, and the induced Riemannian metric is vertically homothetic to that induced by the Sasaki metric.  We say that $\s$ is {\sl $q$-Riemannian\/} if the pullback $\s^*h_{p,q}$ is Riemannian.  This condition, which is independent of $p$, is equivalent to either $\s$ having constant length, or:
\begin{equation}
q\/|\s(x)|^2\geqs-1,
\quad\text{for all $x\in M$.}
\label{qriem}
\end{equation}
Note that \eqref{qriem} allows $\s$ to touch the boundary of the $q$-Riemannian ball bundle.  

\par
In this paper we say that $\s$ is a {\sl harmonic vector field\/} if $\s$ is $(p,q)$-harmonic for some pair $(p,q)$, which we refer to as {\sl metric parameters for $\s$.}  Under the weaker condition that $T_p(\s)$ is pointwise collinear with $\s$ we say that $\s$ is {\sl $p$-preharmonic.}  If the metric parameters for a harmonic $\s$ are unique we will say that $\s$ is {\sl metrically unique.}   The previously mentioned examples show that this need not be the case.  However, with one exception, all the examples considered in this paper turn out to be metrically unique; and the exception has just two pairs of metric parameters.  Nevertheless, it is interesting to learn that metric non-uniqueness can occur for vector fields of non-constant length.  Part of the problem when looking for harmonic vector fields is to identify metric parameters, and to this end the following result of \cite{BLW1} provides some guidance.

\begin{prop}\label{propharm}
Suppose $M$ is compact and $\s$ is a harmonic vector field of non-constant length.  For each $p$ there exists at most one pair of metric parameters $(p,q)$ for $\s$.
Furthermore if $|p|\leqs1$, or $p>1$ and $\|\s\|_\infty\leqs1/\sqrt{p-1}$, then $q<0$, with $q<1-p/2$ if $p\geqs2$.
\end{prop}

Proposition \ref{propharm} does not assert the existence of a lower bound for $q$ on compact manifolds, or that a harmonic vector field is necessarily $q$-Riemannian.  However this turns out to be the case for all the examples considered in this paper.  Interestingly, the contrapositive can be used to provide bigger than expected lower bounds for their image diameter.

\par
A feature of the harmonicity equation \eqref{harmeqn} that is implicit in Proposition \ref{propharm} is its non-invariance under dilations of $TM$.  In fact, if $\s$ is harmonic and $r\in\R$ then unless $r=0,-1$, or the length of $\s$ is constant, it cannot be inferred that $r\s$ is harmonic, even under a change of metric parameters.  Our examples illustrate this very clearly.  Thus harmonic vector fields come with a preferred scale, and determining this is another element of the problem.  It also means that the space of harmonic vector fields is not self-evidently contractible.  However, equation \eqref{harmeqn} is invariant under the action of the isometry group of $(M,g)$.  Thus if $\s$ is $(p,q)$-harmonic then so is any vector field congruent to $\s$, and harmonic vector fields therefore need only be classified up to congruence.  (This could justifiably be considered a basic requirement of the theory).  Continuous families of harmonic vector fields may therefore be easily constructed by the action of $1$-parameter subgroups of isometries.  However in this paper we exhibit a $1$-parameter family of non-congruent harmonic vector fields.  Such families appear in general when $(M,g,J)$ is a K\"ahler manifold, in which case there is a circle action:
$$
e^{it}.\s=(\cos t)\/\s+(\sin t)\/J\s,
$$
with respect to which $\tau_{p,q}$ is equivariant:
$$
\tau_{p,q}(e^{it}.\s)=e^{it}.\tau_{p,q}(\s).
$$ 
Thus, if $\s$ is $(p,q)$-harmonic then so is $e^{it}.\s$ for all $t$.  Borrowing the terminology of classical minimal surface theory, we refer to this $S^1$-orbit of $\s$ as the {\sl associate family\/} of harmonic vector fields, and $J\s$ as the {\sl conjugate\/} harmonic field.

\par
Throughout the paper $M$ is a simply-connected non-flat real space form; ie. the sphere $S^n$ or real hyperbolic space $H^n$.  Our approach is unashamedly extrinsic, which we find provides greater geometric insight.
As far as possible we try to treat the two cases simultaneously, although there are often marked differences between them, and certain situations where the hyperbolic case requires more attention.  Underlying our search for harmonic vector fields are a number of unifying themes.  One is that all vector fields considered can be constructed from {\sl conformal gradient fields.}  Such fields were considered as the prime example in \cite{BLW1}, but for the sphere only, and it was shown that whilst none exist on the $2$-sphere, for each dimension greater than two there exists a unique (up to congruence) harmonic representative, which is also metrically unique.  In Section \ref{secconfgrad} we extend that classification to hyperbolic space.  Notable differences are the appearance of a $2$-dimensional example, and a family of metrically non-unique examples in higher dimensions.  Many of the basic computational elements  set out in Section \ref{secconfgrad} will be re-used in our analyses of more complicated vector fields.  In Section \ref{seckill} we classify all harmonic Killing fields on $M$.  It came as somewhat of a surprise to us to learn that Killing fields can be constructed from conformal gradient fields, and we utilise such decompositions later in the paper.  However in Section \ref{seckill}, whilst decomposition into conformal gradients affords an attractive analysis of harmonic Killing fields on spheres, in the hyperbolic case there are additional complications that make the procedure rather unwieldy; we have therefore chosen an alternative approach that works well for both.  We first show that all harmonic Killing fields on hyperbolic space are infinitesimal rotations, thus eliminating infinitesimal translations and fields of parabolic type.  These, along with all spherical Killing fields, are classified up to congruence by their spectrum of rotational frequencies, which we refer to as the field's {\sl twists,} and we then show that harmonicity demands that these are all equal, or in other words that the field is {\sl balanced.}  Finally, after discounting the Hopf fields on odd-dimensional spheres (which are the only Killing fields of constant length, and well-known to be harmonic \cite{Wie, Wood2}), we show that there is a unique value for the twist, and determine the metric parameters, which are also unique.  In the spherical case it transpires that harmonic Killing fields of non-constant length exist only in dimension four and higher.  However harmonic Killing fields exist on hyperbolic spaces of all dimensions; in particular, on the hyperbolic plane, we note the appearance of the conjugate of the unique (up to congruence) harmonic conformal gradient field.  In Section \ref{seclox} we examine {\sl loxodromic\/} vector fields.  These conformal fields provide a natural deformation between conformal gradient fields and certain types of Killing fields: non-Hopf fields on the sphere, or infinitesimal rotations on hyperbolic space.  On the hyperbolic plane, the associate family of the harmonic vector fields uncovered in Sections \ref{secconfgrad} and \ref{seckill} is loxodromic.  We show that these are in fact the only harmonic loxodromic fields in any dimension.  In Section \ref{secdipole} we study another natural conformal deformation between conformal gradients and Killing fields, that is applicable to infinitesimal hyperbolic translations.  This deformation passes through a {\sl dipole field,} which is of interest in its own right.  However, the only possible harmonic vector fields in the entire deformation are the ``endpoints''; therefore in the hyperbolic case only the conformal gradient field ``endpoint'' can be harmonic, and on the $2$-sphere there are no harmonic vector fields of this type.  In Section \ref{sec2conf} we consolidate results of Sections \ref{seclox} and \ref{secdipole} by classifying all harmonic conformal vector fields in $2$-dimensions, confirming non-existence on the $2$-sphere.  Finally in Section \ref{secquadgrad}, with the $2$-sphere somewhat in mind, we classify all harmonic {\sl quadratic gradient fields\/} on spheres.  In fact, these exist only on odd-dimensional spheres of dimension five and higher.  Each such sphere has a unique (up to congruence) harmonic representative $\s$, which is also metrically unique.  As one might hope, the geometry of $\s$ is pleasingly simple: a pair of equidimensional mutually orthogonal totally geodesic submanifolds of maximal dimension play the role of source and sink, and $\s$ flows along the inter-connecting great circles with a specified maximum speed.

\par
A second unifying theme of the paper is that all our examples of harmonic vector fields turn out to be eigenfunctions of the rough Laplacian.  In this case $\s$ is clearly $0$-preharmonic, and $p$-preharmonic for $p\neq0$ if and only if:
\begin{equation}
\nab{\nabla F}\s=\z\/\s,
\label{preharm}
\end{equation}
for some smooth function $\z\colon M\to\R$, where necessarily: 
\begin{equation}
|\s|^2 \z=|\nabla F|^2.
\label{spinn}
\end{equation}  
Equation \eqref{preharm} may be regarded as a partial integrability condition for the full harmonicity equation \eqref{harmeqn}, and is notable for its independence of the metric parameters.  We say an eigenfunction $\s$ is {\sl preharmonic\/} if \eqref{preharm} holds, and refer to $\z$ as the {\sl spinnaker\/} of $\s$. Use of \eqref{spinn} together with the Weitzenb\"ock identity:
\begin{equation}
\<\nabla^*\nabla\s,\s\>=|\nabla\s|^2+\Delta F
\label{weitz}
\end{equation}
then reduces the harmonicity equation \eqref{harmeqn} to the following PDE for $F$:
\begin{equation}
(p+q+2qF)\Delta F+2p(1+qF)\z+\nu\bigl(1+2(1-p)F\bigr)=0,
\label{eigenharm}
\end{equation}
where $\nu$ is the eigenvalue.

\section{Conformal Gradient Fields}\label{secconfgrad}

Throughout the paper we use the hyperboloid model of hyperbolic space.  
Thus, let $\R^{n,1}$ denote $\R^{n+1}$ equipped with the Lorentzian inner product:
$$
\<x,y\>=x_1y_1+\cdots+x_ny_n-x_{n+1}y_{n+1},
\quad\text{for all $x,y\in\R^{n+1}$.}
$$
Then:
$$
H^n=\{x\in\R^{n,1}:\<x,x\>=-1,\,x_{n+1}>0\}.
$$
For the sphere we use the standard model:
$$
S^n=\{x\in\R^{n+1}:\<x,x\>=1\},
$$
where $\<\text{-},\text{-}\>$ now denotes the Euclidean inner product on $\R^{n+1}$.  In both cases the Riemannian metric $g$ on $M$ is that induced by $\<\text{-},\text{-}\>$, and for notational convenience we will often also denote $g$ by $\<\text{-},\text{-}\>$.  In order to handle spherical and hyperbolic computations simultaneously we introduce the indicator symbol $\e$, which is set to $1$ when $M=S^n$ and $-1$ when $M=H^n$.  Furthermore, we let $\V$ denote ambient $(n+1)$-dimensional Euclidean or Lorentzian space, as appropriate.

\par
Let $a\in\V$ be any vector, and set $\mu=\<a,a\>$.  If $\V$ is Lorentzian then $\mu$ may be negative, in which case we assume without loss of generality that $a$ is future oriented.  Let $\a\colon  M\to\R$ be the restriction of the covector metrically dual to $a$:
\begin{equation}
\a(x)=\<a,x\>,
\quad\text{for all $x\in M$.}
\label{dual}
\end{equation}
The {\sl conformal gradient field\/} $A$ on $M$ with {\sl pole\/} $a$ is then defined $\s=\nabla\a$, the spherical/hyperbolic gradient of $\a$.  We sometimes refer to $\a$ as the {\sl potential\/} of $\s$.  It follows from the linearity of $\a$ that for all $X\in TM$:
\begin{equation}
\<\s,X\>=d\a(X)=\a(X)=\<a,X\>.
\label{confdual}
\end{equation}
Therefore if $\{E_i\}$ is an orthonormal basis of $T_xM$:
\begin{align}
\s(x)
&=\textstyle\sum_i\<\s(x),E_i\>E_i
=\textstyle\sum_i\<a,E_i\>E_i
=a-\e\/\a(x)x,
\label{confield}
\end{align}
since $\{E_1,\dots,E_n,x\}$ is an orthonormal basis of $\V$.  In particular:
\begin{align}
|\s|^2
&=\mu-\e\/\a^2.
\label{conflength}
\end{align}
In the spherical case $\s$ has precisely two zeros (viz. $\pm a/\sqrt{\mu}\,$), and $\mu$ is the maximum value of $|\s|^2$, which is attained on the equatorial hypersphere orthogonal to $a$.  In the hyperbolic case it follows from \eqref{confield} and \eqref{conflength} that $\s$ has unbounded length.  Furthermore $\s$ has a zero only when $a$ is timelike, in which case the integral curves of $\s$ are, as point sets, the geodesic rays from the zero (viz. $a/\sqrt{-\mu}\/$). 
If $a$ is spacelike then the integral curves of $\s$ are the ultraparallel family of geodesics orthogonal to the totally geodesic hypersurface $\uperp a\cap H^n$, which we refer to as the {\sl equator\/} for $\s$.  Furthermore $\mu$ is the minimum value of $|\s|^2$, which is attained on the equator.  If $a$ is lightlike the integral curves of $\s$ are asymptotically parallel geodesics, and $\s\to0$ asymptotically.  In the spherical case, and the hyperbolic cases with $\mu\neq0$, $\s$ is determined up to congruence by $\mu$.  

\par
The covariant derivative of $M$ is characterised by the Gauss formula \cite{ON}:
\begin{equation}
\nab XY=\diffop DXY+\e\<X,Y\>x,
\label{gauss}
\end{equation}
where $D$ denotes the standard directional derivative of $\R^{n+1}$, and $x$ is interpreted as a unit normal to $M$.  Hence by \eqref{confield} and \eqref{confdual}:
\begin{align}
\nab X\s
=-\e\/\a(X)x-\e\/\a(x) X+\e\<X,\s\>x
=-\e\/\a X.
\label{confcov}
\end{align}
From \eqref{confcov} the divergence of $\s$ is:
\begin{equation}
\Div \s
=\textstyle\sum_i\<\nab{E_i}\s,E_i\>
=-\e n\a.
\label{confdiv}
\end{equation}
Furthermore the second covariant derivative is:
\begin{align*}
\nabsq XY\s
&=\nab X\nab Y\s-\nab{\nab XY}\s
=-\e\/\nab X(\a Y)+\e\a\/\nab XY 
=-\e\<\s,X\>Y,
\end{align*}
from which, by \eqref{roughlap}:
\begin{equation}
\nabla^*\nabla \s=\e\/\s.
\label{confrough}
\end{equation}
So $\s$ is an eigenfunction of the rough Laplacian.

\begin{lemma}\label{lemconf}
All conformal gradient fields on a non-flat space form are preharmonic, with spinnaker $\z=\e(\mu-2F)$.
\end{lemma}

\begin{proof}
From \eqref{conflength}:
\begin{equation}
\nabla F=-\e\a\nabla\a=-\e\a\/\s,
\label{confgrad}
\end{equation}
hence by \eqref{confcov}:
\begin{equation*}
\nab{\nabla F}\s=-\e\a\nabla F=\a^2\s.
\label{confspinn}
\end{equation*}
Therefore $\s$ is preharmonic, and $\z$ follows by comparison with \eqref{conflength}. 
\end{proof}

\begin{theorem}\label{thmconf}
Let $\s$ be a non-trivial conformal gradient field on the non-flat space form $M$, with pole $a\in\V$ and $\<a,a\>=\mu$.
\flushpar
{\rm(1)}\quad
If $\mu\geqs0$ {\rm(}ie. $M=S^n$, or $M=H^n$ and $\s$ has no zeros{\rm)} then $\s$ is harmonic if and only if $n>2$ and $\mu=1/(n-2)$.  Furthermore $\s$ is metrically unique, with metric parameters $(p,q)=(n+1,2-n)$.
\flushpar
{\rm(2)}\quad
If $\mu<0$ {\rm(}ie. $M=H^n$ and $\s$ has a zero{\rm)} then $\s$ is harmonic if and only if $\mu=-1$.  Furthermore if $n=2$ then $\s$ is metrically unique, with metric parameters $(3,-1/2)$, and if $n>2$ then $\s$ has precisely two pairs of metric parameters:
$$
\text{\rm(a)\quad} (p,q)=(n+1,1-n+1/n);
\qquad
\text{\rm(b)\quad} (p,q)=(1/(2-n),0).
$$
\end{theorem}

\begin{proof}
By \eqref{confgrad}, \eqref{confdiv} and \eqref{conflength}:
$$
\Delta F
=-\Div\nabla F=\e\Div(\a\/\s)
=\e|\s|^2-n\a^2 
=\e(2(n+1)F-n\mu),
$$
where we have used the general Riemannian identity:
\begin{equation}
\Div(fX)=f\Div X+df(X).
\label{divident}
\end{equation}
Using this and Lemma \ref{lemconf}, the harmonicity equation \eqref{eigenharm} may therefore be written:
\begin{equation*}
(p+q+2qF)(2(n+1)F-n\mu)
+2p(1+qF)(\mu-2F)
+2(1-p)F+1=0.
\end{equation*}
This is polynomial in $F$, the coefficients of which yield the following three equations:
\begin{align}
(n+1-p)q&=0, 
\label{confquad} \\
(n-2)p+(n+1)q+(p-n)q\mu&=-1, 
\label{conflin} \\
((2-n)p-nq)\mu&=-1.
\label{confconst}
\end{align}
If $q=0$ then $(2-n)p=1$ and $\mu=-1$.  Otherwise $p=n+1$, whence adding \eqref{conflin} and \eqref{confconst} yields:
$$
(n-2+q)(1+\mu)=0.
$$
If $\mu=-1$ then $nq=1+n-n^2$, by \eqref{confconst}.  Otherwise $q=2-n$, and $(n-2)\mu=1$ by \eqref{confconst}.
\end{proof}

In view of Proposition \ref{propharm}, the metric parameters $(1/(2-n),0)$ cannot occur on the $n$-sphere.  All other harmonic conformal gradient fields $\s$ have metric parameter $q<0$.  As remarked in \cite{BLW1}, in the spherical case $\s$ is $q$-Riemannian, with $\s(M)$ touching the boundary of the Riemannian ball bundle on the equator of $\s$.  In the hyperbolic case the opposite occurs: if $\s$ has no zeros then $\s(M)$ lies entirely outside the Riemannian ball bundle, and touches its boundary on the equator of $\s$.  
The case $n=2$ is distinctive, for by \cite[Theorem 2.8 and Theorem 3.5]{BLW2} the metric parameters $(3,-1/2)$ endow the fibres (resp. total space) of the Riemannian ball bundle with positive sectional (resp. scalar) curvature. 

\section{Killing Fields}\label{seckill}

\par
Let $\s$ be a Killing field on $M$.  Then $\s$ is an eigenfunction of the rough Laplacian; for by a general curvature identity for Killing fields:
\begin{equation}
\nabla^*\nabla\s=\Ric(\s)=\e(n-1)\s.
\label{killrough}
\end{equation}
For computational purposes we view $\s$ as the restriction to $M$ of a skew-symmetric linear transformation $A$ of $\V$; thus:
\begin{equation}
\s(x)=A(x),
\quad\text{for all $x\in M$.}
\label{killfield}
\end{equation}
In the spherical case, the following result may be deduced without computation from the Divergence Theorem.  However the hyperbolic case requires the Laplacian $\Delta F$, which in both cases will be required for later results.

\begin{prop}\label{propkill1}
Every harmonic Killing field $\s$ on $M$ is preharmonic.
\end{prop}

\begin{proof}
We need to show that, if non-trivial, $\s$ cannot have metric parameter $p=0$.
It follows from \eqref{harmeqn}, \eqref{harmterms} and \eqref{killrough} that $\s$ is $(0,q)$-harmonic if and only if:
\begin{equation}
q\/\Delta F=\e(1-n).
\label{killpzero}
\end{equation}
To compute $\Delta F$ we first apply the Gauss formula \eqref{gauss} to \eqref{killfield}:
\begin{equation}
\nab X\s
=A(X)+\e\<X,\s\>x
=A(X)-\e\<A(X),x\>x.
\label{killdiff}
\end{equation}
Then, summing over an orthonormal tangent frame $\{E_i\}$:
\begin{align}
\nabla F
&=\<\nab{E_i}\s,\s\>E_i
=\<A(E_i),A(x)\>E_i
\notag \\
&=-A^2(x)+\e\<A^2(x),x\>x
=-A^2(x)-\e|\s|^2 x.
\label{killgrad}
\end{align}
Using \eqref{gauss} once again:
\begin{align}
\nab X(\nabla F)
&=\diffop DX(\nabla F)+\e\<X,\nabla F\>x \notag \\
&=-A^2(X)+\e\<A^2(X),x\>x-\e|\s|^2X,
\intertext{from which:}
\Delta F=-\Div\nabla F
&=-\<\nab{E_i}(\nabla F),E_i\>
=\e n|\s|^2-\<A(E_i),A(E_i)\> \notag \\
&=\e(n+1)|\s|^2-|A|^2,
\label{killlap}
\end{align}
where $|A|^2$ denotes the Lorentzian norm when $\e=-1$ (see \eqref{lorentzip} below; however it suffices to note that this is constant). 
If $\s$ is $(0,q)$-harmonic, comparison of \eqref{killpzero} with \eqref{killlap} shows that $\s$ has constant length, which contradicts \eqref{killpzero}.
\end{proof}

\par
Before proceeding, we pause to summarise some relevant geometric features of Killing fields on space forms.  In the spherical case there exists an orthonormal basis $\{e_i\}$ of Euclidean $\R^{n+1}$, a positive integer $r\leqs (n+1)/2$ and real numbers $\w_1\geqs\w_2\geqs\cdots\geqs\w_r>0$ such that for all $1\leqs i\leqs r$ and $2r<j\leqs n+1$:
\begin{equation}
A(e_{2i-1})=\w_ie_{2i},
\qquad
A(e_{2i})=-\w_ie_{2i-1},
\qquad
A(e_j)=0.
\label{normal1}
\end{equation}
The string $(\w_1,\dots,\w_r)$ determines $A$ (as a linear endomorphism), hence $\s$ (as a vector field), up to congruence.  Geometrically, the flow of $\s$ is a superposition of rotations around the $r$ orthogonal great circles cut out by the $2$-planes $e_{2i-1}\wedge e_{2i}$ of $\R^{n+1}$, with angular frequencies $\w_1,\dots,\w_r$.  We refer to $r$ as the {\sl rotational rank\/} of $\s$, and the $\w_i$ as the {\sl twists\/} of $\s$.  In particular, $\w_1=\|\s\|_\infty$.  If $2r\leqs n$ then $\ker A$ intersects $S^n$ in a great $(n-2r)$-sphere, which we refer to as the {\sl axis\/} of $\s$. 
If $\w_1=\dots=\w_r=\w$ then we say that $\s$ is {\sl balanced\/} with twist $\w$.  In this case $\s$ is congruent to $\w\Sigma_r$ where:
\begin{equation}
\Sigma_r(x)=(-x_2,x_1,\dots,-x_{2r},x_{2r-1},0,\dots,0).
\label{hopf}
\end{equation}
If $n$ is odd and $2r=n+1$ then $\Sigma_r$ is the standard Hopf field; otherwise $\Sigma_r$ restricts to a Hopf field on the great $(2r-1)$-sphere orthogonal to its axis.

\par
In the hyperbolic case, fix $w\in H^n$ and set $v=A(w)=\s(w)\in T_wH^n$.  Let $R=R_w\colon\V\to\V$ be the orthogonal projection of $A$ into $T_wH^n$, defined:
\begin{align*}
R(u)
&=A(u+\<u,w\>w)+\<A(u+\<u,w\>w),w\>w \\
&=A(u)+\<u,w\>v-\<u,v\>w,
\end{align*}
for all $u\in\V$.
Then $R(w)=0$ and the restriction of $R$ to $T_wH^n$ is a skew-symmetric endomorphism of a Euclidean vector space, which may therefore be put into normal form \eqref{normal1}, with $2r\leqs n$.  Now define $T=T_w=A-R_w$:
\begin{equation}
T(u)=\<u,v\>w-\<u,w\>v.
\label{killtrans}
\end{equation}
Since $R$ and $T$ are skew-symmetric transformations of $\V$, their restrictions to $H^n$ are also Killing vector fields, which we refer to as the {\sl rotational\/} (resp. {\sl translational\/}) {\sl part\/} of $\s$ at $w$.  We also refer to $r$ as the {\sl rotational rank\/} of $\s$ at $w$, and  $\w_1,\dots,\w_r$  as the {\sl twists of $\s$ at $w$.}  The dependence of the rotational rank and twists  on $w$ is constrained as follows.  Let $\<\text{-},\text{-}\>_L$ denote the indefinite inner product induced on the space of Lorentzian skew-symmetric endomorphisms of $\V$:
$$
\<A_1,A_2\>_L=\trace \bigl(A_1\circ (A_2)^\dag\bigr),
$$
where $A^\dag$ denotes the Lorentz adjoint.  In terms of a Lorentz-orthonormal frame containing $w$:
\begin{equation}
\<A_1,A_2\>_L=\textstyle\sum_i\<A_1(e_i),A_2(e_i)\>-\<A_1(w),A_2(w)\>,
\label{lorentzip}
\end{equation}
and it follows that:
\begin{align*}
|A|_L^{\;2}=|R|_L^{\;2}-|T|_L^{\;2}
=2r\w^2-2\tau^2,
\end{align*}
where $\tau=|v|$ and $\w$ is the quadratic mean twist: $\w^2=(\w_1^{\,2}+\cdots+\w_r^{\,2})/r$.  (The second equation follows from the invariance of $\<\text{-},\text{-}\>_L$ under Lorentzian congruences.)  Thus the function $r\w^2-\tau^2$ is constant over $H^n$.  Its sign determines the nature of $\s$.  If positive, then $\s$ is an {\sl infinitesimal rotation.}  In this case $\ker(A)$ is timelike, and $\ker(A)\cap H^n$ is a totally geodesic submanifold of $H^n$ which we refer to as the {\sl axis of rotation.}  The translational part of $\s$ vanishes at all points on the axis.  Furthermore the rotational rank and twists of $\s$ are invariant along the axis, and determine $\s$ up to congruence.  In particular, if $\w_1=\dots=\w_r=\w$ along the axis then we say that $\s$ is {\sl balanced\/} with twist $\w$ and rotational rank $r$, in which case $\s$ is congruent to $\w\Sigma_r$ with $\Sigma_r$ defined by \eqref{hopf}.  However since $2r<n+1$, in contrast to the spherical case $\Sigma_r$ has at least one zero, and has precisely one zero when $2r=n$.  We refer to $\Sigma_r$, in both the spherical and hyperbolic cases, as a {\sl generalised Hopf field.}  If $r\w^2\leqs\tau^2$ then $\ker(A)$ is spacelike or lightlike, so $\s$ has no zeros.  When the inequality is strict, $\s$ is an {\sl infinitesimal translation.}  In this case there is a unique geodesic of $H^n$ that is an integral curve of $\s$, which we refer to as the {\sl axis of translation.}  The rotational part of $\s$ vanishes at all points on the axis, and the length of $\s$ along the axis, which is necessarily constant, determines $\s$ up to congruence.  Finally, if $r\w^2=\tau^2$ then $\s$ is of {\sl parabolic type.}  

\medskip
Our next step, in the light of Proposition \ref{propkill1}, is to explore the implications of preharmonicity.  We therefore compute, using \eqref{killdiff} and \eqref{killgrad}:
\begin{align}
\nab{\nabla F}\s
&=A(\nabla F)-\e\<A(\nabla F),x\>x
=-A^3(x)-\e|\s|^2\s.
\label{killcov}
\end{align}
(Note that since $A^3$ is also skew-symmetric, its restriction to $M$ is indeed a (Killing) vector field on $M$.)  
Thus $\s$ is preharmonic if and only if there exists a smooth function $\lambda\colon M\to\R$ such that:
\begin{equation}
A^3(x)-\lambda(x)A(x)=0,
\quad\text{for all $x\in M$.}
\label{killnec1}
\end{equation}
Differentiation of \eqref{killnec1} yields:
\begin{equation}
A^3(X)-\lambda(x)A(X)=d\lambda(X)A(x),
\quad
\text{for all $X\in T_xM$.}
\label{killnec2}
\end{equation}
It follows from \eqref{killnec1} and \eqref{killnec2} that for each $x\in M$ the linear map $A^3-\lambda(x)A$ has rank at most one.  But non-trivial skew-symmetric transformations of $\V$ have rank at least two;  so $A^3-\lambda(x)A=0$ for all $x\in M$.  Since $A(x)$ vanishes at worst on the intersection of $M$ with an $(n-1)$-dimensional subspace of $\R^{n+1}$---which has measure zero\/---\/it then follows from \eqref{killnec2} that $d\lambda=0$ generically, hence by continuity identically.  Therefore by connectedness $\lambda$ is constant.  We conclude that $\s$ is preharmonic if and only if:
\begin{equation}
A^3-\lambda A=0,
\quad\text{for some $\lambda\in\R$.}
\label{evalue}
\end{equation}
Comparison with \eqref{killcov} shows that the spinnaker is:
\begin{equation}
\z=-(\lambda+\e|\s|^2).
\label{killspinn}
\end{equation}
We first explore the implications of this for the sphere.  

\begin{prop}\label{propkill2}
Let $\s$ be a Killing field on $S^n$.  Then $\s$ is preharmonic if and only if
$\s$ is balanced.  
\end{prop}

\begin{proof}
It follows from \eqref{evalue} that $\s$ is preharmonic if and only if $\im(A)$ is an eigensubspace of the symmetric operator $A^2$.
Since $A$ is skew-symmetric with respect to a Euclidean metric, $\im(A)$ and $\ker(A)$ are (orthogonal) complementary subspaces of $\V$.  In particular $A$, and hence $A^2$, restricts to an isomorphism of $\im(A)$, which is therefore the direct sum of all the non-zero eigenspaces of $A^2$.  Thus $\s$ is preharmonic if and only if $A^2$ has a unique non-zero eigenvalue.  From the normal form \eqref{normal1} the non-zero eigenvalues of $A^2$ are $-\w_i^{\,2}$, $1\leqs i\leqs r$.  Hence $\s$ is preharmonic if and only if $\s$ is balanced.  
\end{proof}

\begin{corollary}\label{corkill1}
Let $\s$ be a Killing field of maximal rotational rank on an odd-dimensional sphere.  Then $\s$ is preharmonic if and only if $\s$ is a  constant multiple of a Hopf field.  Therefore if $\s$ is preharmonic then $\s$ is harmonic.
\end{corollary}

We now scrutinise hyperbolic space.  Choose $w\in H^n$ and decompose $A=R+T$, to obtain the rotational and translational parts of $\s$ at $w$.  Set $v=A(w)\in T_wH^n$.  Our next result is diametrically opposite to the conclusion of Corollary \ref{corkill1}.

\begin{prop}\label{propkill3}
Let $\s$ be an infinitesimal translation of $H^n$.  Then $\s$ is preharmonic, but not harmonic.
\end{prop}

\begin{proof}
Suppose $w$ lies on the axis of $\s$.  Then $A=T$ and it follows from \eqref{killtrans} that for all $u\in\V$:
\begin{align}
T^3(u)
&=T^2(\<u,v\>w-\<u,w\>v) 
=T(\<u,v\>v-\<u,w\>|v|^2w) \notag \\
&=\<u,v\>\tau^2w-\<u,w\>\tau^2v
=\tau^2\/T(u),
\label{tcubed}
\end{align}
where $\tau=|v|$.  Thus \eqref{evalue} is satisfied with $\lambda=\tau^2$, so $\s$ is preharmonic with:
\begin{equation*}
\z=2F-\tau^2,
\end{equation*}
by \eqref{killspinn}.
Furthermore by \eqref{killlap}:
$$
\Delta F=2(\tau^2-(n+1)F).
$$
Harmonicity equation \eqref{eigenharm} is therefore polynomial in the non-constant smooth function $F$ (indeed, no non-trivial Killing field on a manifold of strictly negative Ricci curvature has constant length), and gathering coefficients yields:
\begin{align*}
(n+1-p)q&=0, \\
(n+1+(p-2)\tau^2)q&=1-n, \\
1-n+2q\tau^2&=0.
\end{align*}
It follows (from the second or third equation) that $q\neq0$, therefore (from the first) $p=n+1$.  Adding the second and third equations then yields $1+\tau^2=0$, contradicting the fact that $v$ is spacelike.
\end{proof}

\begin{prop}\label{propkill4}
Let $\s$ be an infinitesimal rotation of $H^n$.  Then $\s$ is preharmonic if and only if $\s$ is balanced.
\end{prop}

\begin{proof}
We may orthogonally decompose $\V=L\oplus T_wH^n$ where $L\subset\V$ is the line through $w$.  Suppose $w$ lies on the axis of $\s$.  Then $A=R$, hence $L\subset\ker(A)$ and the restriction of $A$ to $T_wH^n$ is a skew-symmetric transformation of a Euclidean vector space.  The argument of Proposition \ref{propkill2} then applies.
\end{proof}

We now show that there are no harmonic Killing fields of parabolic type.

\begin{prop}\label{propkill5}
Let $\s$ be a Killing field on $H^n$.
If $\s$ is harmonic then $\s$ is an infinitesimal rotation.
\end{prop}

\begin{proof}
Suppose $\s$ is harmonic.
Since $v\in\im(A)$ and $\s$ is preharmonic it follows from \eqref{evalue} that $A^2(v)=\lambda v$.  Now:
$$
RT(v)=R(\tau^2w)=0,
$$
and by \eqref{killtrans} and \eqref{tcubed}:
\begin{align*}
TR(v)
=TA(v)-T^2(v)
&=\<A(v),v\>w-\<A(v),w\>v-\tau^2v=0.
\end{align*}
Therefore:
$$
A^2(v)=R^2(v)+T^2(v)=R^2(v)+\tau^2v.
$$
Since $R$ may be regarded (by restriction) as a skew-symmetric endomorphism of the Euclidean vector space $T_wH^n$, there is an orthogonal decomposition $v=\hat v+\check v$ where $\hat v\in\im(R)$ and $\check v\in\ker(R)$.  Then:
$$
A^2(v)=R^2(\hat v)+\tau^2\hat v+\tau^2\check v.
$$
Therefore, since $R^2(\hat v)\in\im(R)$, $v$ lies in the $\lambda$-eigenspace of $A^2$ if and only if:
\begin{equation}
R^2(\hat v)=(\lambda-\tau^2)\hat v
\quad\text{and}\quad
(\lambda-\tau^2)\check v=0.
\label{psquared}
\end{equation}

\medskip\noindent
{\bf Case 1.}\;
Suppose $\hat v=0$.  Then for all $u\in\im(R)$:
$$
A(u)=R(u)+T(u)=R(u)+\<u,v\>w-\<u,w\>v=R(u).
$$
Since $R$ restricts to an isomorphism of $\im(R)$, this also shows that $\im(R)\subset\im(A)$.  Proposition \ref{propkill3} ensures $R\neq0$, so it follows from \eqref{evalue} that $\lambda$ is a non-zero eigenvalue of $R^2$.  Then $\lambda<0$, since $R$ is Euclidean skew-symmetric, hence $\check v=0$ by \eqref{psquared}.  Therefore $v=0$.  Thus $\s$ is an infinitesimal rotation, with $w$ on the axis.  

\medskip\noindent
{\bf Case 2.}\;
Suppose $\hat v\neq0$.  Since $R^2$ restricts to an isomorphism of $\im(R)$ it follows from \eqref{psquared} that $\lambda-\tau^2\neq0$, hence $\check v=0$.  So $v\in\im(R)$ is a non-null eigenvector of $R^2$.  Therefore $v=v_1e_1+\cdots+ v_{2r}e_{2r}$ with respect to the orthonormal basis of \eqref{normal1}.  We now introduce the index subset: 
$$
I=\{i\in\{1,\dots,r\}:(v_{2i-1})^2+(v_{2i})^2\neq0\}.
$$
It follows that all the $\w_i$ with $i\in I$ are equal; say $\w_i=\rho$.  Thus $R^2(v)=-\rho^2v$, and comparison with \eqref{psquared} yields:
$$
\lambda=\tau^2-\rho^2.
$$
Let $J$ be the complementary index subset.  Thus $j\in J$ if and only if $v_{2j-1}=v_{2j}=0$, and it follows that: 
$$
T(e_{2j-1})=0=T(e_{2j}).
$$ 
Hence:
$$
A(e_{2j-1})=R(e_{2j-1})=\w_j\/e_{2j},
\qquad
A(e_{2j})=R(e_{2j})=-\w_j\/e_{2j-1}.
$$
So $e_{2j-1},e_{2j}\in\im(A)$ and:
$$
A^2(e_{2j-1})=-\w_j^{\,2}\,e_{2j-1},
\qquad
A^2(e_{2j})=-\w_j^{\,2}\,e_{2j}.
$$
Therefore $\w_j^{\,2}=-\lambda=\rho^2-\tau^2$ for all $j\in J$, by \eqref{evalue}.  In particular, $\rho>\tau$.  Since $R^2$ has precisely two non-zero eigenvalues there is an orthogonal splitting:
$$
\im(R)=V\oplus W,
$$
where $V$ (resp. $W$) is the eigenspace for eigenvalue $-\rho^2$ (resp. $\tau^2-\rho^2$); indeed, $V$ (resp. $W$) is spanned by $\{e_{2i-1},e_{2i}:i\in I\}$ (resp. $\{e_{2j-1},e_{2j}:j\in J\}$), and $v\in V$.  If $i\in I$ then: 
$$
T(e_{2i-1})=v_{2i-1}w,
\qquad
T(e_{2i})=v_{2i}w,
$$
hence:
$$
A(e_{2i-1})=\rho\/e_{2i}+v_{2i-1}\/w,
\qquad
A(e_{2i})=-\rho\/e_{2i-1}+v_{2i}\/w.
$$
Therefore:
$$
A(v_{2i}\/e_{2i-1}-v_{2i-1}\/e_{2i})
=\rho(v_{2i-1}\/e_{2i-1}+v_{2i}\/e_{2i}),
$$
so the vector $a_i=v_{2i-1}\/e_{2i-1}+v _{2i}\/e_{2i}\in V$ also belongs to $\im(A)$.  Now:
\begin{align*}
A^2(a_i)
&=A\big(v_{2i-1}(\rho\/e_{2i}+v_{2i-1}\/w)
+v_{2i}(-\rho\/e_{2i-1}+v_{2i}\/w)\big)
=-\rho^2\/a_i+|a_i|^2v.
\end{align*}
Since $a_i$ is an eigenvector of $A^2$ with eigenvalue $\lambda=\tau^2-\rho^2$, we must have $v=a_i$.  We conclude that $V$ is $2$-dimensional.  In particular, this implies:
$$
r\w^2-\tau^2
=\rho^2+(r-1)(\rho^2-\tau^2)-\tau^2
=r(\rho^2-\tau^2)>0.
$$
Therefore $\s$ is again an infinitesimal rotation, but with $w$ no longer on the axis. 
\end{proof}

The spherical and hyperbolic cases can henceforward be treated simultaneously.
For fixed $r$ with $2r<n+1$ we introduce the following {\sl twist equation:}
\begin{equation}
2ck\/\w^4+\e(2nk-c)\w^2+1-n=0,
\label{twist}
\end{equation}
where $c=n+1-2r$ and $k=r-1$.  If $r=1$ then \eqref{twist} has solutions only when $\e=-1$, with unique positive root $\w_0=1$.  If $r>1$ the discriminant $\Delta$ of \eqref{twist}, viewed as quadratic in $\w^2$, is strictly positive, and the signs of the coefficients show that there is precisely one positive root $\w_0$,
which we refer to as the {\sl optimal twist.}  We claim that if $\e=1$ (resp. $\e=-1$) then $\w_0<1$ (resp. $\w_0>1$).  For, from the following upper bound:
$$
\Delta=(2nk-c)^2+8ck(n-1)=(2nk+c)^2-8ck<(2nk+c)^2,
$$
we obtain for $\e=1$:
\begin{equation}
\w_0^{\,2}
<\frac{c-2nk+2nk+c}{4ck}=\frac{2c}{4ck}
=\frac{1}{2(r-1)}.
\label{spherebound}
\end{equation}
So in the spherical case there is in fact the sharper estimate $\w_0<1/\sqrt2$.  On the other hand, since $n>c$ there is the following lower bound:
$$
\Delta>(2ck-c)^2+8c^2k=c^2(2k+1)^2,
$$
from which if $\e=-1$ we obtain:
\begin{equation}
\w_0^{\,2}
>\frac{2nk-c+c(2k+1)}{4ck}
=\frac{n+c}{2c}=1+\frac{n-c}{2c}.
\label{hypbound}
\end{equation}

\par
It was noted (Corollary \ref{corkill1}) that on odd-dimensional spheres the harmonic Killing fields of maximal rank are constant multiples of Hopf fields, which are precisely the spherical Killing fields of constant length.  By contrast, there are no non-trivial hyperbolic Killing fields of constant length. The following result deals with all other cases; ie. Killing fields of non-constant length.  Recall the generalised Hopf field $\Sigma_r$ defined in \eqref{hopf}.

\begin{theorem}\label{thmkill}
Let $\s$ be a Killing field of non-constant length on the non-flat space form $M$.  Then $\s$ is harmonic if and only if $\s$ is congruent to $\w_0\/\Sigma_r$ for some $0<r<(n+1)/2$, with $r>1$ if $M=S^n$.  
Furthermore $\s$ is metrically unique, with metric parameters $p=n+1$ and: 
\begin{equation}
q=
\begin{cases}
(1-n)/2, 
&\text{if } r=1, \\
2(1-r)\displaystyle\frac{\w_0^{\,2}}{\w_0^{\,2}+\e},
&\text{if } r>1.
\tag{\dag}
\label{dag}
\end{cases}
\end{equation}
\end{theorem}

\begin{proof}
If $\s$ is harmonic then it follows from Propositions \ref{propkill1}, \ref{propkill2}, \ref{propkill4} and \ref{propkill5} that in the spherical case $\s$ is balanced, and in the hyperbolic case $\s$ is a balanced infinitesimal rotation.  Then equation \eqref{evalue} holds with $\lambda=-\w^2$, where $\w$ is the twist.  By \eqref{killspinn} the spinnaker of $\s$ is:
$$
\z=\w^2-2\e F.
$$
Furthermore by \eqref{killlap}:
$$
\Delta F=2(\e(n+1)F-r\w^2).
$$
Harmonicity equation \eqref{eigenharm} is therefore polynomial in the non-constant function $F$, and the coefficients yield:
\begin{align}
(n+1-p)q&=0,
\label{quadcoeff} \\
(\e(n+1)+(p-2r)\w^2)q&=\e(1-n),
\label{lincoeff} \\
2((r-1)p+rq)\w^2&=\e(n-1).
\label{constcoeff}
\end{align}
It follows from \eqref{lincoeff} that $q\neq0$, whereupon $p=n+1$ by \eqref{quadcoeff}.  Adding \eqref{lincoeff} and \eqref{constcoeff} then yields:
\begin{equation}
2(r-1)\w^2+q(\w^2+\e)=0.
\label{killfinal}
\end{equation}
If $r=1$ then $\e=-1$ and $\w=1$, and it follows from \eqref{constcoeff} that $q=(1-n)/2$.  If $r>1$ then \eqref{killfinal} rearranges to \eqref{dag} for $q$, substitution of which into \eqref{lincoeff} yields the twist equation \eqref{twist} for $\w$.
\end{proof}

It follows from \eqref{dag} that all harmonic Killing fields $\s$ of non-constant length have metric parameter $q<0$.  (Recall that in the hyperbolic case $\w_0>1$ when $r>1$.)  In the spherical case, plugging the upper bound \eqref{spherebound} for $\w_0$ into \eqref{dag} yields the following lower bound for $q$:
\begin{equation}
q=-\frac{2k}{1+1/\w_0^{\,2}}
>-\frac{2k}{1+2k} 
=-\frac{2r-2}{2r-1}.
\label{sphereq}
\end{equation}
Thus $-1<q<0$, in contrast to harmonic conformal gradient fields, where $q\leqs-1$.  Since $\w_0^{\,2}<1$ it follows that $q\w_0^{\,2}>-1$, which since $\w_0=\|\s\|_\infty$ shows that $\s$ is $q$-Riemannian \eqref{qriem}.  In fact, when \eqref{sphereq} is combined with \eqref{spherebound} the following sharper estimate is obtained:
$$
q\/\w_0^{\,2}
>\frac{q}{2r-2}
>-\frac{1}{2r-1}\geqs-\frac13,
\quad\text{since $r\geqs2$.}
$$  
So $\s(M)$ sits comfortably inside the $q$-Riemannian ball bundle.  On the other hand, bearing in mind that $n>3$ if $\s$ has non-constant length: 
$$
1-\frac{p}{2}=1-\frac{n+1}{2}<-1<q.
$$  
It therefore follows from Proposition \ref{propharm} that $\w_0^{\,2}>1/n$.  We conclude that:
\begin{equation}
\frac{1}{n}<\w_0^{\,2}<\frac{1}{2r-2}.
\label{spheretwist}
\end{equation}
These bounds are tightest when $n$ is even and $\s$ has maximal rotational rank $r=n/2$.  For example, applying \eqref{sphereq} and \eqref{spheretwist} to the unique (up to congruence) non-trivial harmonic Killing field on $S^4$ yields:
$$
-1/3<q<0,
\qquad
1/4<\w_0^{\,2}<1/2,
$$
in comparison to the precise values:  
$$
q=(\sqrt{73}-13)/8,
\qquad
\w_0^{\,2}=(\sqrt{73}-7)/4.
$$

In the hyperbolic case, if $n\geqs4$ and $\s$ has rotational rank $r>1$ then since $\w_0>1$ it follows from \eqref{dag} that $q<2(1-r)$.  In particular, $q<-2$.  From \eqref{hypbound}:
$$
\frac{\w_0^{\,2}}{\w_0^{\,2}-1}
=\frac{1}{1-1/\w_0^{\,2}}<\frac{n+c}{n-c},
$$
which by \eqref{dag} yields a lower bound on $q$:
\begin{equation}
2\d(1-r)<q<2(1-r),
\quad\text{where $\d=\frac{n+c}{n-c}>1$.}
\label{hypq}
\end{equation}
These bounds are again tightest when $n$ is even and $r=n/2$.  For example, applying \eqref{hypq} and \eqref{hypbound} to the unique (up to congruence) harmonic Killing field on $H^4$ with $r=2$ yields:
$$
-10/3<q<-2,
\qquad
\w_0^{\,2}>5/2,
$$
in comparison to the precise values:
$$
q=-(\sqrt{73}+13)/8,
\qquad
\w_0^{\,2}=(\sqrt{73}+7)/4.
$$

\section{Loxodromic Fields}\label{seclox}

Suppose $a,b\in\V$ and let $\a,\b$ be the restrictions to $M$ of the corresponding metrically dual linear functionals.  First define for all $x\in M$:
\begin{equation}
K(x)=\a(x)b-\b(x)a.
\label{elkilltrans}
\end{equation}
Then $K$ is the restriction to $M$ of a skew-symmetric transformation of $\V$, hence a Killing field.  If $A,B$ are the conformal gradient fields with poles $a,b$ then by \eqref{confield}:
\begin{align}
K(x)
&=\a\/B-\b\/A.
\label{elkill}
\end{align}
We refer to $K$ as the {\sl elementary Killing field\/} determined by the pair $(a,b)$.  From \eqref{confcov}, the covariant derivative of $K$ is particularly simple:
\begin{align}
\nab XK
&=\<A,X\>B-\<B,X\>A.
\label{elkillcov}
\end{align}
It is also useful to note from \eqref{confield} the following generalisation of \eqref{conflength}:
\begin{equation}
\<A,B\>=\<a,b\>-\e\a\b.
\label{confprod}
\end{equation}
\par
Now let $\RR=\{a_1,b_1,\dots,a_r,b_r\}$ be a spacelike orthonormal subset of $\V$.  Let $K_i$ be the elementary Killing field determined by $(a_i,b_i)$, and define:
$$
R=\textstyle\sum_i\w_iK_i,
\quad\w_i>0.
$$ 
From Section \ref{seckill}, every Killing field on $S^n$, or infinitesimal rotation on $H^n$, of rotational rank $r$ has such an expansion.  A vector field $\s$ on $M$ is then said to be {\sl loxodromic\/} if:
$$
\s=R+C,
$$
where $C$ is a non-trivial conformal gradient field whose pole $c$ is orthogonal to $\RR$.  We say that $\s$ is {\sl properly loxodromic\/} if $R$ is balanced and $n=2r$; in the hyperbolic case this implies that $c$ is timelike.  Then $R$ and $C$ have precisely the same zeros, and the angle between $\s(x)$ and $C(x)$ is constant.  
If $n=2$ then every loxodromic field is properly loxodromic.

\par
Let $\g$ be the restriction to $M$ of the linear functional metrically dual to $c$.  It follows from \eqref{elkill} and \eqref{confprod} that:
\begin{equation*}
\<K_i,C\>
=\a_i\<B_i,C\>-\b_i\<A_i,C\>
=-\e\a_i\b_i\g+\e\b_i\a_i\g=0.
\end{equation*}
Furthermore since $\RR$ is spacelike:
\begin{align}
\<K_i,K_j\>
&=\d_{ij}(\a_i^{\,2}+\b_i^{\,2}),
\label{elkillprod}
\end{align}
where $\d_{ij}$ is the Kronecker symbol.  So if $\s$ is loxodromic then (summing over $i$):
\begin{align}
|\s|^2
&=\w_i^{\,2}|K_i|^2+|C|^2
=\w_i^{\,2}(\a_i^{\,2}+\b_i^{\,2})+\mu-\e\g^2,
\label{loxlength}
\end{align}
where $\mu=\<c,c\>$.
Note that $\s$ vanishes precisely on the axis of $R$; ie. the totally geodesic submanifold $M\cap\uperp{\RR}\subset M$ of codimension $2r$.

\begin{lemma}\label{lemlox}
Let $\s$ be a loxodromic vector field on a non-flat space form.  If $\s$ is $p$-preharmonic for some $p$ then $\s$ is properly loxodromic.
\end{lemma}

\begin{proof}
From \eqref{killrough} and \eqref{confrough}:
\begin{equation}
\nabla^*\nabla\s
=\e(n-1)R+\e C.
\label{loxrough}
\end{equation}
Therefore $\s$ is $0$-preharmonic if and only if $n=2$, in which case $\s$ is properly loxodromic.  Now suppose $p\neq0$.  Introducing the gradient vector fields:
$$
G_i=\a_iA_i+\b_i B_i,
$$
it follows from \eqref{loxlength} that (summing over $i$):
\begin{equation}
\nabla F
=\w_i^{\,2}G_i-\e\g C.
\label{loxgrad}
\end{equation}
From \eqref{confprod}:
\begin{align*}
\<G_i,K_j\>
&=\d_{ij}(\a_j\b_i-\a_i\b_j)=0.
\end{align*}
Therefore $\nabla F$ is orthogonal to $R$.  
From \eqref{elkillcov} and \eqref{confcov} (summing over $i$):
$$
\nab{\nabla F}\s
=\w_i\<A_i,\nabla F\>B_i-\w_i\<B_i,\nabla F\>A_i-\e\g\nabla F. 
$$
Then by \eqref{loxgrad} and \eqref{confprod} (summing over all repeated indices):
\begin{align}
\nab{\nabla F}\s+\e\g\nabla F
&=\w_i\/\w_j^{\,2}\bigl(\a_j(\d_{ij}-\e\a_i\a_j)-\e\a_i\b_j^{\,2}\bigr)B_i
+\w_i\a_i\g^2 B_i \notag \\
&\qquad
-\w_i\/\w_j^{\,2}\bigl(\b_j(\d_{ij}-\e\b_i\b_j)-\e\b_i\a_j^{\,2}\bigr)A_i
-\w_i\b_i\g^2 A_i \notag \\
&=\e(\mu+\e\w_i^{\,2}-|\s|^2)\w_iK_i,
\quad\text{by \eqref{loxlength}.}
\label{loxderiv}
\end{align}
Since $\nabla F$ is orthogonal to $R$, it follows from \eqref{loxrough}, \eqref{loxderiv} and \eqref{harmterms} that
if $\s$ is $p$-preharmonic then $R$ is balanced and $\nabla F$ is pointwise collinear with $C$.  The latter is equivalent to the pointwise collinearity of $G=\sum_iG_i$ and $C$.  Using \eqref{confprod} we first note that for all $i,j$:
\begin{align*}
\<G_i,G_j\>
&=(\a_i^{\,2}+\b_i^{\,2})(\d_{ij}-\e(\a_j^{\,2}+\b_j^{\,2})),
\end{align*}
and then calculate (summing over repeated indices):
\begin{align*}
|C\wedge G|^2
&=|C|^2|G|^2-\<C,G\>^2 \notag \\
&=(\a_i^{\,2}+\b_i^{\,2})\bigl(\mu-\e\g^2-\e\mu(\a_j^{\,2}+\b_j^{\,2})\bigr)  \\
&=|K_i|^2(|C|^2-\e\mu |K_j|^2). \notag
\end{align*}
Since $K_i$ vanishes only on a set of measure zero, if $\e=-1$ then $C\wedge G$ vanishes only if $\mu<0$; ie. $c$ is timelike.  So there exist vectors $d_1,\dots,d_{n-2r}\in\V$, which in the hyperbolic case are spacelike, such that $\{a_1,b_1,\dots,a_r,b_r,d_1,\dots,d_{n-2r},c/\sqrt{\e\mu}\}$ is an orthonormal basis of $\V$.  Hence, if $\d_i\colon M\to\R$ is the restriction of the covector metrically dual to $d_i$ then there is the quadratic relation:
\begin{equation}
\textstyle\sum_{j=1}^r(\a_j^{\,2}+\b_j^{\,2})
+\d_1^{\,2}+\cdots+\d_{n-2r}^{\,2}+\g^2/\mu=\e,
\label{loxquad}
\end{equation}
which rearranges to:
$$
|C|^2-\e\mu\textstyle\sum_j|K_j|^2
=\e\mu(\d_1^{\,2}+\cdots+\d_{n-2r}^{\,2}).
$$
Therefore if $C\wedge G=0$ then $n=2r$.
\end{proof}

\begin{prop}\label{proplox}
Let $\s$ be a loxodromic vector field on the non-flat space form $M^n$.  Then $\s$ is $p$-preharmonic for some $p$ if and only if $n=2$.  In this case $\s$ is an eigenfunction of the rough Laplacian, and $\s$ is preharmonic with spinnaker:
\begin{equation*}
\z=\e(\mu+\e\w^2-|\s|^2).
\end{equation*}
\end{prop}

\begin{proof}
Suppose that $\s$ is properly loxodromic.  Since $n=2r$, relation \eqref{loxquad} reduces to (summing over $i$):
\begin{equation}
\mu(\a_i^{\,2}+\b_i^{\,2})+\g^2=\e\mu,
\label{loxquadtwor}
\end{equation}
which when differentiated yields:
\begin{equation}
\mu G+\g C=0.
\label{loxquaddiff}
\end{equation}
Furthermore since $R$ is balanced \eqref{loxderiv} simplifies:
\begin{equation*}
\nab{\nabla F}\s=\e(\mu+\e\w^2-|\s|^2)R-\e\g\nabla F.
\end{equation*}
Application of \eqref{loxquadtwor} to \eqref{loxlength} yields:
\begin{equation}
\mu|\s|^2=(\mu+\e\w^2)(\mu-\e\g^2).
\label{loxproplength}
\end{equation}
Also application of \eqref{loxquaddiff} to \eqref{loxgrad} yields:
\begin{equation}
\mu\nabla F
=-\e(\mu+\e\w^2)\g C,
\label{loxpropgrad}
\end{equation}
and then by \eqref{loxproplength}:
$$
\mu\g\nabla F
=-\e(\mu+\e\w^2)\g^2C
=-\mu(\mu+\e\w^2-|\s|^2)C.
$$
This yields the further simplification:
\begin{equation*}
\nab{\nabla F}\s
=\e(\mu+\e\w^2-|\s|^2)\/\s.
\end{equation*}
It then follows from \eqref{loxrough} that $\s$ is $p$-preharmonic if and only if $n=2$.  The result is now a consequence of Lemma \ref{lemlox}. 
\end{proof}

Define vector fields $\s_0$ and $\s_1$ on $H^2$ as follows:
\begin{gather*}
\s_0(x_1,x_2,x_3)=(-x_2,x_1,0), \\
\s_1(x_1,x_2,x_3)=(-x_1x_3,-x_2x_3,1-x_3^{\,2}).
\end{gather*}
Then $\s_0$ (resp. $\s_1$) is an infinitesimal rotation (resp. conformal gradient) with axis (resp. pole) $(0,0,1)$, and any loxodromic field on $H^2$ is congruent to $\w\s_0+\sqrt{-\mu}\/\s_1$.  By Theorems \ref{thmconf} and \ref{thmkill}, up to congruence $\s_0$ (resp. $\s_1$) is the unique harmonic Killing (resp. conformal gradient) field on $H^2$, and both fields are metrically unique, with metric parameters $(p,q)=(3,-1/2)$.  Furthermore, when $H^2$ is regarded as a Riemann surface $\s_0$ and $\s_1$ are conjugate.  The associate family:
$$
\LL=\{\sin t\,\s_0+\cos t\,\s_1:t\in\R\}
$$
is therefore also harmonic, and loxodromic (with the exception of $\s_0$ and $\s_1$, which our definition excludes).

\begin{theorem}\label{thmlox}
Let $\s$ be a loxodromic vector field on the non-flat space form $M$.  Then $\s$ is harmonic if and only if $M=H^2$ and $\s$ is congruent to an element of $\LL$.  Furthermore $\s$ is metrically unique, with metric parameters $(p,q)=(3,-1/2)$. 
\end{theorem}

\begin{proof}
In Proposition \ref{prop2conf} we will show that when $n=2$ the harmonicity equation \eqref{harmeqn} for a conformal field simplifies:
$$
\e(1+2F)(1+2qF)+2q((p-2)F-1)\z=0,
$$
where $\z$ is the spinnaker.  It follows from Proposition \ref{proplox} that this equation is polynomial in $F$, with leading coefficient $(3-p)q$.  Therefore $p=3$, since there is clearly no solution if $q=0$.
The terms of lower degree then yield:
$$
2q(\mu+\e\w^2)=1=-(\mu+\e\w^2),
$$
from which $q=-1/2$ and $\e=-1$.  Then $\mu<0$ and $\w^2-\mu=1$.
\end{proof}

\section{Conformal Dipole Deformations}\label{secdipole}

First let $M=S^n$.  Pick any $w\in S^n$ and $a\in T_w S^n$.  The {\sl Pontryagin vector field\/} \cite{Ped} determined by this point and tangent vector is constructed by parallel translating $a$ along the geodesics from $w$, to obtain a vector field of constant length with a singularity at $-w$.  The singularity may be smoothed off to a zero by a natural scaling factor\/---\/the square of the cosine of half the geodesic distance from $a$\/---\/resulting in a smooth vector field $\d$ with a point dipole at $-w$.  Analysis of this procedure, which we omit for brevity, produces the following expression: 
$$
\d(x)=\tfrac12\bigl((1+\psi(x))a-\a(x)(x+w)\bigr),
\quad\text{for all $x\in S^n$,}
$$
where $\psi,\a$ are the metric duals of $w,a$.  It follows from \eqref{elkill} and \eqref{confield} that:
$$
2\d=A-T,
$$
where $A$ is the conformal gradient with pole $a$, and $T$ is the elementary Killing field determined by the pair $(a,w)$.  If $W$ is the conformal gradient with pole $w$ then:
\begin{equation}
T=\a W-\psi A.
\label{diptrans}
\end{equation}
The dipole may be repositioned from $-w$ to $w$ by reversing the sign of $a$,
and any constant rescaling of the field may also be incorporated into $a$. 
We therefore define, for $M=S^n$ or $M=H^n$, the {\sl dipole field\/} determined by $w\in M$ and $a\in T_wM$ to be:
\begin{equation*}
\s=T+\e A.
\label{dipfield}
\end{equation*}
The geodesic through $w$ tangent to $a$ is the {\sl dipole axis.}  In the hyperbolic case, $A$ is a conformal gradient field without zeros, and comparison of \eqref{diptrans} and \eqref{killtrans} shows that $T$ is an infinitesimal translation whose axis coincides with that of the dipole and meets the equator of $A$ orthogonally at $w$.  

\par
More generally, suppose $a\in T_wM$ is a unit vector and define for $r,\tau\in\R$:
\begin{equation}
\s=\tau\/T+r A.
\label{dipdef}
\end{equation}
We refer to $\s$ as a {\sl dipole deformation field.}  Bearing in mind that $\<w,w\>=\e$ and $\<a,w\>=0$ it follows from \eqref{confprod}:
\begin{align}
|\s|^2
&=\tau^2|T|^2+2r\tau\<T,A\>+r^2|A|^2 \notag \\
&=(\tau\psi-r)^2+\e(\tau^2-r^2)\a^2.
\label{diplength}
\end{align}  
If $r=\tau$ and $M=S^n$, or $r=-\tau$ and $M=H^n$, then $\s$ is a dipole field and it follows from \eqref{diplength} that $\s$ has a single zero, at $w$, as expected.  Otherwise, depending on the relative magnitudes of $\tau$ and $r$, the point dipole at $w$ splits into a pair of zeros on the dipole axis, with midpoint $w$; or the zero set of $\s$ is a codimension-1 submanifold of the equator of $A$ (ie. codimension-2 in $M$) with two connected components, placed symmetrically with respect to $w$.

\begin{lemma}\label{lemdip}
Let $\s$ be a dipole deformation field on the non-flat space form $M^n$, that is neither a conformal gradient nor a Killing field.  Then $\s$ is $p$-preharmonic for some $p$ if and only if $n=2$.  In this case $\s$ is an eigenfunction of the rough Laplacian, and $\s$ is preharmonic with spinnaker:
$$
\z=\e(r^2+\tau^2-2r\tau\psi-|\s|^2).
$$ 
\end{lemma}

\begin{proof}
By assumption $\tau\neq0$ and $r\neq0$.  From \eqref{diplength}:
\begin{equation}
\nabla F=\tau(\tau\psi-r)W+\e(\tau^2-r^2)\a A.
\label{dipgrad}
\end{equation}
Therefore by \eqref{confcov}, \eqref{elkillcov}, \eqref{confprod} and \eqref{diptrans}:
\begin{align*}
\nab{\nabla F}\s+\e r\a\/\nabla F
&=\tau\<A,\nabla F\>W-\tau\<W,\nabla F\>A \\
&=\e\tau^2(r-\tau\psi)(A+\psi T)
+\tau(\tau^2-r^2)\a(\e W-\a T ),
\end{align*}
which after further extensive calculation yields:
\begin{align}
\nab{\nabla F}\s
&=(r^2-\tau^2)\a^2\s
+\e\tau^2(1-\psi^2)\s.
\label{dipcov}
\end{align}
Now from \eqref{confrough} and \eqref{killrough}:
\begin{equation}
\nabla^*\nabla\s=\e(n-1)\tau\/T+rA.
\label{diprough}
\end{equation}
Hence $\s$ is $p$-preharmonic precisely when $n=2$.  The expression for $\z$ follows by comparison of \eqref{dipcov} with \eqref{diplength}.
\end{proof}

\begin{theorem}\label{thmdip}
Let $\s$ be a non-trivial dipole deformation field on a non-flat space form.  Then $\s$ is harmonic if and only if $\s$ is a harmonic Killing field or a harmonic conformal gradient field.
\end{theorem}

\begin{proof}
By Lemma \ref{lemdip} it suffices to consider the case $n=2$.  Anticipating Proposition \ref{prop2conf}, the harmonicity equation then simplifies:
$$
\e(1+2F)(1+2qF)+2q((p-2)F-1)\z=0,
$$
where $\z$ is the spinnaker.  It follows from \eqref{diplength} and Lemma \ref{lemdip} that this is a polynomial $P(\a,\psi)=0$.  The highest order (quartic) terms are:
$$
P_4(\a,\psi)
=(3-p)q\bigl(\e(\tau^2-r^2)^2\a^4
+2\tau^2(\tau^2-r^2)\a^2\psi^2
+\e\tau^4\psi^4\bigr).
$$
Since $q\neq0$ and $\s$ is non-trivial, $p=3$.  The cubic terms then reduce to:
$$
P_3(\a,\psi)=6qr\tau\bigl((r^2-\tau^2)\a^2\psi-\e\tau^2\psi^3\bigr),
$$
the vanishing of which forces $r=0$ or $\tau=0$.
\end{proof}

\section{Conformal Vector Fields in Dimension Two}\label{sec2conf}

\par
Suppose that $\s$ is a conformal vector field on $M=S^2$ or $M=H^2$.  Let $J$ be one of the two K\"ahler structures on $(M,g)$ (the choice is immaterial).  The conformal transformations of $M$ are then precisely its biholomorphic mappings, so $\s$ is a holomorphic vector field.  Furthermore, $\s$ may be decomposed as the sum of a Killing field and a conformal gradient field, so it follows from \eqref{confrough} and \eqref{killrough} that $\s$ is an eigenfunction of the rough Laplacian, with eigenvalue $\e$.  Our first result uses these observations to recast the harmonicity equation \eqref{harmeqn} for $\s$.

\begin{prop}\label{prop2conf}
Let $\s$ be a conformal vector field on $M^2$.  Then $\s$ is preharmonic; and $\s$ is $(p,q)$-harmonic if and only if:
$$
\e(1+2F)(1+2qF)+2q((p-2)F-1)\z=0,
$$
where $\z$ is the spinnaker of $\s$.
In particular, if $\s$ is $(p,q)$-harmonic then $q\neq0$.
\end{prop}

\begin{proof}
Since $\s$ is a holomorphic vector field we have:
$$
\nab{\s}(J\s)
=\nab{J\s}\s+[\s,J\s]
=\nab{J\s}\s+J[\s,\s]
=\nab{J\s}\s,
$$
because $\diffop L{\sigma}J=0$.
The components of $\nabla\s$ may therefore be obtained:
\begin{gather*}
\<\nab{\s}\s,\s\>
=\<\nabla F,\s\>
=\<\nab{\s}(J\s),J\s\>
=\<\nab{J\s}\s,J\s\>, \\
\<\nabla F,J\s\>
=\<\s,\nab{J\s}\s\>
=\<\s,\nab{\s}(J\s)\>
=-\<\nab{\s}\s,J\s\>.
\end{gather*}
Now:
$$
2F\,\nabla F
=\<\nabla F,\s\>\s+\<\nabla F,J\s\>J\s.
$$
Hence:
\begin{align*}
2F\<\nab{\nabla F}\s,J\s\>
&=\<\nabla F,\s\>\<\nab\s \s,J\s\>
+\<\nabla F,J\s\>\<\nab{J\s}\s,J\s\>
=0.
\end{align*}
Because the zeros of $J\s$ are discrete, this implies that $\nab{\nabla F}\s$ is pointwise collinear with $\s$.  Therefore, since also $\s$ is an eigenfunction of the rough Laplacian, $\s$ is preharmonic.  We now observe:
\begin{align*}
2F\,|\nabla F|^2
&=\<\nabla F,\s\>^2+\<\nabla F,J\s\>^2 \\
&=\tfrac12\bigl(\<\nab \s\s,\s\>^2+\<\nab{\s}\s,J\s\>^2
+\<\nab{J\s}\s,\s\>^2+\<\nab{J\s}\s,J\s\>^2\bigr) \\
&=2F^2\,|\nabla\s|^2.
\end{align*}
By comparison with \eqref{spinn} the spinnaker of $\s$ must therefore be:
\begin{equation*}
\z=\tfrac12\/|\nabla\s|^2.
\end{equation*}
It then follows from the Weitzenb\"ock identity \eqref{weitz} that:
\begin{equation}
\Delta F=\e|\s|^2-|\nabla\s|^2=2(\e F-\z),
\label{spinnex}
\end{equation}
which yields the stated simplification of harmonicity equation \eqref{eigenharm}.
The equation clearly cannot be satisfied if $q=0$.
\end{proof}

We will now represent: 
$$
\s=K+C,
$$ 
where $K$ is a Killing field and $C$ a conformal gradient.  To do so, first choose $w\in M$ and an orthonormal basis $(a,b)$ of $T_wM$.  Let $\a,\b,\psi$ denote the restrictions to $M$ of the linear forms on $\V$ metrically dual to $a,b,w$ respectively, and define Killing vector fields on $M$:
$$
R=\a B-\b A,
\qquad
T=\a W-\psi A,
$$
where $A,B,W$ are the conformal gradients with poles $a,b,w$ respectively.  In the hyperbolic case $R$ is an infinitesimal rotation about $w$, and $T$ is an infinitesimal translation through $w$, therefore:
$$
K=\w R+\tau\/T,
$$
for some $\w,\tau\geqs0$.  In the spherical case the situation is simpler, for by choosing $w$ to lie on the axis of $K$ it may be assumed that $\tau=0$.  Now let $c$ be the pole of $C$, and let $\g$ be its metrically dual covector.  We locate $c$ cylindrically:
\begin{equation*}
c=rsa+rtb+hw,
\end{equation*}
where $h\in\R$, $r\geqs0$ and $s^2+t^2=1$.  Finally we note the quadratic relation:
\begin{equation}
Q(\a,\b,\psi)=\a^2+\b^2+\e\psi^2-\e=0,
\label{2confrel}
\end{equation}
which expresses the equation of $M$ with respect to the ``coordinates'' $\a,\b,\psi$.

\begin{lemma}\label{lem2conf1}
If $\s$ is a conformal vector field on $M^2$ then:
$$
|\s|^2=\tau^2+r^2+\e(\w^2+h^2)
+2(\w rt+\e\tau h)\a-2rs(\w\b+\tau\psi)
-\e(\w\psi-\e\tau\/\b)^2-\e\g^2.
$$
\end{lemma}

\begin{proof}
We have:
\begin{equation*}
|\s|^2
=\w^2|R|^2+\tau^2|T|^2+|C|^2
+2\w\tau\<R,T\>+2\w\<R,C\>+2\tau\<T,C\>.
\end{equation*}
Each inner product of conformal gradients may be computed from \eqref{confprod}:
\begin{gather*}
|R|^2=\a^2+\b^2,
\qquad
|T|^2=\e\a^2+\psi^2,
\qquad
|C|^2=r^2+\e h^2, \\
\<R,T\>=\b\psi,
\qquad
\<R,C\>=r(t\a-s\b),
\qquad
\<T,C\>=\e h\a-rs\psi,
\end{gather*}
and the result follows after applying the relation \eqref{2confrel}.
\end{proof}

\begin{lemma}\label{lem2conf2}
If $\s$ is a conformal field on $M^2$ then the spinnaker of $\s$ is:
$$
\z=(\w\psi-\e\tau\b)^2+\g^2.
$$
\end{lemma}

\begin{proof}
From Lemma \ref{lem2conf1}:
\begin{align*}
\nabla F
&=(\w rt+\e\tau h)A-2rs(\w B+\tau W)
-\e(\w\psi-\e\tau\b)(\w W-\e\tau B)-\e\g C.
\end{align*}
Therefore by \eqref{confdiv} and \eqref{divident}:
\begin{align*}
\Delta F
&=2\e(\w rt+\e\tau h)\a-2\e rs(\w\b+\tau\psi)-2(\w\psi-\e\tau\b)^2-2\g^2 \\
&\qquad
+\e|\w W-\e\tau B|^2+\e|C|^2 \\
&=\e|\s|^2-2(\w\psi-\e\tau\b)^2-2\g^2,
\end{align*}
by \eqref{confprod} and comparison with Lemma \ref{lem2conf1}.  The result follows from \eqref{spinnex}.
\end{proof}

It follows from Proposition \ref{prop2conf} and Lemmas \ref{lem2conf1} and \ref{lem2conf2} that $\s$ is harmonic if and only if:
$$
P(\a,\b,\psi)=0,
$$
where $P$ is a quartic polynomial.  This is equivalent to the algebraic problem of $P$ vanishing modulo the polynomial $Q$ of \eqref{2confrel}; thus if $P_k$ denotes the terms of $P$ that are homogeneous of degree $k$ we require $P_k=(QS)_k$ for all $0\leqs k\leqs4$, for some (necessarily quadratic) polynomial $S$.  Now:
$$
P_4(\a,\b,\psi)=\e q(3-p)\z^2,
\qquad
P_3(\a,\b,\psi)=q(p-4)\eta\/\z,
$$
where $\eta$ denotes the linear part of $|\s|^2$, and $q\neq0$.  On the other hand:
\begin{equation}
(QS)_4(\a,\b,\psi)=(\a^2+\b^2+\e\psi^2)S_2(\a,\b,\psi),
\label{2confmultquart}
\end{equation}
and because $Q_1=0$:
\begin{equation}
(QS)_3(\a,\b,\psi)=(\a^2+\b^2+\e\psi^2)S_1(\a,\b,\psi).
\label{2confmultcub}
\end{equation}
We consider the spherical and hyperbolic cases separately.

\begin{theorem}\label{thm2conf1}
No non-trivial conformal field on the $2$-sphere is harmonic.
\end{theorem}

\begin{proof}
In addition to the simplification $\tau=0$, by rotating the basis $(a,b)$ in $T_wS^2$ if necessary it may be assumed that $s=0$ and $t=1$.  Then: 
$$
\z=\w^2\psi^2+\g^2
=(\w^2+h^2)\psi^2+2rh\b\psi+r^2\b^2.
$$
In view of Theorem \ref{thmconf} we assume $\w\neq0$.  Furthermore, $\s$ is loxodromic (resp. an element of a dipole deformation) when $r=0$ (resp. $h=0$), so by Theorem \ref{thmlox} (resp. Theorem \ref{thmdip}) we assume $r\neq0$ and $h\neq0$.
Since $\z$ is independent of $\a$, no match of $P_4$ with \eqref{2confmultquart} is possible unless $P_4$ vanishes identically; hence $p=3$.  Now $\eta=2\w r\a$, therefore $P_3$ has a non-vanishing $\a\b\psi$ term, which means that no match with \eqref{2confmultcub} is possible unless $P_3$ vanishes identically, which is also impossible.
\end{proof}

\begin{theorem}\label{thm2conf2}
If $\s$ is a non-trivial harmonic conformal field on the hyperbolic plane then $\s$ is congruent to an element of the associate family $\LL$.
\end{theorem}

\begin{proof}
If $\w=h=0$ then $\s$ belongs to a dipole deformation; so by Theorem \ref{thmdip} we assume $\w\neq0$ or $h\neq0$.  The coefficient of $\psi^2$ in $\z$ is $\w^2+h^2$, so $S_2$ contains the term $-(\w^2+h^2)^2\psi^2$.  The coefficient of $\a^2$ in $\z$ is $r^2t^2$, so $S_2$ also contains the term $r^4t^4\a^2$.  Therefore $(QS)_4$ contains the term:
$$
-(r^4t^4+(\w^2+h^2)^2)\a^2\psi^2,
$$
whose coefficient is strictly negative.  However the coefficient of $\a^2\psi^2$ in $P_4$ is $4r^2t^2h^2$.  It follows that $P_4$ vanishes identically, hence $p=3$.  Now, inspecting the coefficients of $\a^3,\b^3$ and $\psi^3$ in $P_3$ yields:
$$
S_1
=2q\bigl((\tau h-\w rt)r^2s^2\a
+\w rs(\tau^2+r^2t^2)\b
-\tau rs(\w^2+h^2)\psi\bigr).
$$
However the coefficients of $\a\psi^2$ and $\b\psi^2$ in $P_3$ are, respectively:
$$
2q(\tau h-\w rt)(\w^2+h^2),
\qquad
2q\/\w rs(\w^2+h^2),
$$
and comparison with the corresponding coefficients in $(QS)_3$ yields equations:
$$
(\tau h-\w rt)(\w^2+h^2+r^2s^2)=0=\w rs(\w^2+\tau^2+r^2t^2+h^2),
$$
which reduce to:
$$
\tau h-\w rt=0=\w rs.
$$
If $\w=0$ then $h\neq0$, so $\tau=0$, and $\s$ is therefore a conformal gradient.  
If $r=0$ then either $h=0$ in which case $\s$ is a Killing field, or $\tau=0$ in which case $\s$ is loxodromic.  Finally, if $s=0$ then: 
\begin{equation}
\w^2 r^2=\tau^2 h^2,
\label{2confharm1}
\end{equation} 
and the lower degree terms of $P$ come into play.  

\par
The equation for $\s$ to be harmonic has now reduced to:
$$
2q(F-1)\z-(1+2F)(1+2qF)=0,
$$
with:
$$
2F=\tau^2-\w^2+r^2-h^2+\z=\theta+\z,
\quad\text{say.}
$$
After some gathering of terms:
$$
P_2(\a,\b,\psi)=-(1+(\theta+3)q)\z
$$
Since $\z$ is now independent of $\a$, $P_2$ is not a multiple of $Q_2$ and therefore vanishes identically:
\begin{equation}
(\theta+3)q=-1.
\label{2confharm2}
\end{equation}
Now $P$ has no linear terms, so it remains to collate the constants:
$$
P_0(\a,\b,\psi)=-(1+\theta)(1+q\theta).
$$
The only solution of $P_0=0$ consistent with \eqref{2confharm2} is:
\begin{equation}
\theta=-1.
\label{2confharm3}
\end{equation}
If $\s$ is an infinitesimal translation or of parabolic type (ie. $\tau^2\geqs\w^2$) then it follows from \eqref{2confharm1} that $\theta\geqs0$, contradicting \eqref{2confharm3}.  If $\s$ is an infinitesimal rotation then choosing for convenience $w$ to be the axis, so that $\tau=0$, it follows from \eqref{2confharm1} that $r=0$; thus $\s$ is again loxodromic.  The result now follows from Theorem \ref{thmlox}.
\end{proof}

\section{Quadratic Gradient Fields on Spheres}\label{secquadgrad}

Let $M=S^n$ with $n>1$.  Let $Q\colon\V\to\V$ be a symmetric linear transformation, and let $\xi\colon M\to\R$ be the restriction of the associated quadratic form:
$$
\xi(x)=\<Q(x),x\>,
\quad\text{for all $x\in M$.}
$$
Now define: 
$$
\s=\tfrac12\nabla\xi.
$$  
If $(a_1,\dots,a_{n+1})$ is an orthonormal $Q$-eigenbasis of $\V$:
$$
Q(a_i)=\lambda_ia_i,
\quad\lambda_i\in\R,
$$
and $\a_i\colon S^n\to\R$ are the restrictions of the metrically dual covectors \eqref{dual}, then:
$$
\xi=\textstyle\sum_i\lambda_i\/\a_i^{\,2}.
$$ 
Therefore:
\begin{equation}
\s=\textstyle\sum_i\lambda_i\a_i A_i,
\label{quadrep}
\end{equation}
where $A_i$ is the conformal gradient field with pole $a_i$.  Applying \eqref{confield} to \eqref{quadrep} yields the following coordinate-free representation:
\begin{align*}
\s(x)
&=\textstyle\sum_i\lambda_i\a_i(x)(a_i-\a_i(x)x)
=Q(x)-\xi(x)x.
\end{align*}
In particular, this shows that the zeros of $\s$ are precisely the (unit) eigenvectors of $Q$.  Thus, $\s$ has at least $2n+2$ zeros, and these are isolated if and only if the eigenvalues of $Q$ are simple.  It should also be noted that $Q$ is not uniquely associated to $\s$, because of the relation:
\begin{equation*}
\a_1^{\,2}+\cdots+\a_{n+1}^{\,2}=1.
\end{equation*}
In particular, the eigenvalues $\lambda_i$ do not determine $\s$ up to congruence.  The conformal representation \eqref{quadrep} is also non-unique, because of the differentiated relation:
\begin{equation}
\a_1A_1+\cdots+\a_{n+1}A_{n+1}=0.
\label{sumgrads}
\end{equation}
However, \eqref{quadrep} does provide computational assistance,  using results of Section \ref{secconfgrad}.  We first compute the divergence of $\s$ (summing over $i$):
\begin{align}
\Div(\s)
&=\lambda_i(\a_i\Div A_i+|A_i|^2),
\quad\text{by \eqref{divident}} \notag \\
&=\lambda_i(1-(n+1)\a_i^{\,2}),
\quad\text{by \eqref{confdiv} and \eqref{conflength}} \notag \\
&=\trace(Q)-(n+1)\xi.
\label{quaddiv}
\end{align}
Also, using the general Riemannian identity:
$$
\nabla^*\nabla(fX)=f\,\nabla^*\nabla X-2\/\nab{\nabla f}X+(\Delta f)X,
$$
we compute using \eqref{confrough}, \eqref{confcov} and \eqref{confdiv} (summing over $i$):
\begin{align}
\nabla^*\nabla\s
&=\lambda_i\a_i A_i+2\lambda_i\a_iA_i+n\lambda_i\a_i A_i
=(n+3)\s.
\label{quadrough}
\end{align}
Thus $\s$ is an eigenfunction of the rough Laplacian.  It follows from \eqref{quadrough} and \eqref{harmterms} that when $p=0$ the harmonicity equation \eqref{harmeqn} reads:
$$
q\/\Delta F+n+3=0,
$$
which contradicts the Divergence Theorem.  Therefore no non-trivial harmonic $\s$ has metric parameter $p=0$.  Hence if $\s$ is harmonic then $\s$ is preharmonic.

\par
Now for any positive integer $m$ let $\xi_m\colon M\to\R$ be the quadratic form associated to $Q^m=Q\circ\cdots\circ Q$ ($m$ iterations):
$$
\xi_m=\textstyle\sum_i(\lambda_i)^{m}\a_i^{\,2},
$$
and let $\s_m$ be the corresponding gradient field:
\begin{equation}
\s_m=\tfrac12\nabla\xi_m
=\textstyle\sum_i(\lambda_i)^{m}\a_iA_i,
\label{itquadrep}
\end{equation}
with coordinate-free expression:
\begin{equation}
\s_m(x)=Q^m(x)-\xi_m(x)x.
\label{itquadfield}
\end{equation}
It follows from \eqref{confprod} that for all $k,m\in\N$ (summing over $i,j$):
\begin{align}
\<\s_k,\s_m\>
&=(\lambda_i)^k(\lambda_j)^m\a_i\a_j\<A_i,A_j\> \notag \\
&=(\lambda_i)^{k+m}\a_i^{\,2}
-(\lambda_i)^k\a_i^{\,2}(\lambda_j)^m\a_j^{\,2}
=\xi_{k+m}-\xi_k\/\xi_m.
\label{quadprod}
\end{align}
In particular:
\begin{equation}
|\s|^2=\xi_2-\xi^2.
\label{quadlength}
\end{equation}
Then:
\begin{equation}
\nabla F=\s_2-2\xi\s.
\label{quadgrad}
\end{equation}  
We note that in addition to vanishing at the zeros of $\s$, $\nabla F$ also vanishes at the midpoints $a_{ij}=(a_i+a_j)/\sqrt2$, and by \eqref{quadlength}:
\begin{align*}
|\s(a_{ij})|^2
&=\tfrac12(\lambda_i^{\,2}+\lambda_j^{\,2})-\tfrac14(\lambda_i+\lambda_j)^2 
=\tfrac14(\lambda_i-\lambda_j)^2.
\end{align*}
Therefore:
\begin{equation}
\|\s\|_\infty=\max_{i,j}\tfrac12|\lambda_i-\lambda_j|.
\label{quadsupnorm}
\end{equation}
Of course, if all eigenvalues of $Q$ are equal then $\s$ is trivial.

\begin{prop}\label{propquad}
Let $\s$ be a non-trivial quadratic gradient field on $S^n$.  Then $\s$ is preharmonic if and only if $Q$ has precisely two distinct eigenvalues.
\end{prop}

\begin{proof}
It follows from \eqref{quadgrad}, \eqref{quadrep} and \eqref{confcov} that (summing over $i,j$):
\begin{align}
\allowdisplaybreaks
\nab{\nabla F}\s
&=\lambda_i\<\nabla F,A_i\>A_i
-\lambda_i\a_i^{\,2}\,\nabla F \notag \\
&=\lambda_i\/\lambda_j^{\,2}\a_j\<A_i,A_j\>A_i
-2\xi\/\lambda_i\lambda_j\a_j\<A_i,A_j\>A_i
-\xi\,\nabla F \notag \\
\intertext{which by \eqref{confprod}:}
&=\s_3-3\/\xi\s_2+(4\xi^2-\xi_2)\s.
\label{quad1}
\end{align}
It therefore follows from \eqref{quad1} that $\s$ is preharmonic if and only if:
\begin{equation}
\s_3-3\/\xi\s_2=f\/\s,
\label{quadcoll}
\end{equation}  
for some smooth function $f\colon M\to\R$.  Since $\s$ is non-trivial there exist $j,k$ such that $\lambda_j\neq\lambda_k$.  Set $x=a_j\in M$ and $X=a_k\in T_x M$.  Covariant differentiation of equation \eqref{quadcoll} along $X$ yields:
$$
\nab X\s_3-3\lambda_j\/\nab X\s_2=f(a_j)\/\nab X\s,
$$
bearing in mind that $\s(a_j)=0$. 
From the conformal representation \eqref{itquadrep}, using \eqref{confcov} and \eqref{confield} we obtain for all $m\in\N$:
\begin{align*}
\nab X\s_m
&=\textstyle\sum_i
(\lambda_i)^m\bigl(\<X,A_i(a_j)\>A_i(a_j)-\d_{ij} X\bigr)  
=((\lambda_k)^m-(\lambda_j)^m)a_k.
\end{align*}
Therefore if $\s$ is preharmonic:
$$
\lambda_k^{\,3}-\lambda_j^{\,3}-3\lambda_j(\lambda_k^{\,2}-\lambda_j^{\,2})
=(\lambda_k-\lambda_j)\/f(a_j),
$$
and cancelling the non-zero common factor yields:
$$
f(a_j)
=\lambda_k^{\,2}-2\/\lambda_j\/\lambda_k-2\/\lambda_j^{\,2}. 
$$
Now if $\lambda_\ell$ is a third distinct eigenvalue then:
$$
f(a_j)=\lambda_\ell^{\,2}-2\/\lambda_j\lambda_\ell-2\lambda_j^{\,2},
$$
and eliminating $f(a_j)$ yields:
$$
\lambda_k+\lambda_\ell=2\lambda_j.
$$
By symmetry (ie. taking $x=a_k$ and $X=a_j$):
$$
\lambda_j+\lambda_\ell=2\lambda_k,
$$
which forces the contradiction $\lambda_j=\lambda_k$.  

\par
Conversely, if $Q$ has precisely two distinct eigenvalues then it follows from \eqref{quadrep} and \eqref{sumgrads} that $\s$ may be represented:
\begin{equation}
\s=\lambda(\a_1A_1+\cdots+\a_rA_r),
\label{quad2grad}
\end{equation}
for some $r<n+1$, where $\lambda$ is the difference of the eigenvalues.  The corresponding potential is:
\begin{equation}
\xi=\lambda(\a_1^{\,2}+\cdots+\a_r^{\,2}).
\label{quad2pot}
\end{equation}
Then for all $m\in\N$:
\begin{equation}
\xi_m=\lambda^{m-1}\xi,
\quad
\s_m=\lambda^{m-1}\s,
\label{quad2iter}
\end{equation}
and \eqref{quad1} therefore reduces to:
\begin{equation}
\nab{\nabla F}\s=(\lambda-2\xi)^2\s.
\label{quad2cov}
\end{equation}
Thus $\s$ is indeed preharmonic.
\end{proof}

The quadratic gradient field $\s$ on $S^n$ represented by \eqref{quad2grad} is congruent to $\lambda\Sigma_r$ where:
$$
\Sigma_r(x)=((1-s_r)x_1,\dots,(1-s_r)x_r,
-s_r\/x_{r+1},\dots,-s_r\/x_{n+1}),
$$
for all $x=(x_1,\dots,x_{n+1})\in S^n$, using the abbreviation:
$$
s_r=x_1^{\,2}+\cdots+x_r^{\,2}.
$$
The integral curves of $\s$ are great circle arcs from a great $(n-r)$-sphere to the orthogonal great $(r-1)$-sphere; and by \eqref{quadsupnorm} the maximum speed is $|\lambda|/2$, which is attained midway along each arc.  For $r\geqs 3$ let $\lambda_0$ be the unique positive root of:
\begin{equation}
(r-2)\lambda^4+2(r^2-5)\lambda^2-8(r+1)=0.
\label{quadmu}
\end{equation}
If $\Delta$ is the discriminant of \eqref{quadmu} then one obtains by elementary means:
\begin{equation*}
4(r^2-2)^2<\Delta<4(r^2-1)^2,
\end{equation*}
provided $r\geqs4$, with the upper bound in fact valid if $r\geqs3$.  These yield bounds:
\begin{equation}
\frac{3}{r-2}<\lambda_0^{\,2}<\frac{4}{r-2}.
\label{quadmubounds}
\end{equation}

\begin{theorem}\label{thmquad}
A non-trivial quadratic gradient field $\s$ on $S^n$ is harmonic if and only if $n\geqs5$ is odd and $\s$ is congruent to $\lambda_0\/\Sigma_r$ with $r=(n+1)/2$. Furthermore $\s$ is metrically unique, with $p=r+1$ and: 
\begin{equation}
2q=\frac{(2-r)(1+r)}{1+r+(\lambda_0/2)^2}.
\label{ddag}
\tag{\ddag}
\end{equation}
\end{theorem}

\begin{proof}
We noted earlier that if $\s$ is harmonic then $\s$ is preharmonic.  Therefore
by Proposition \ref{propquad} it may be assumed that $\xi$ and $\s$ are of the form \eqref{quad2pot} and \eqref{quad2grad} respectively.  It then follows from \eqref{quad2iter} that \eqref{quadlength} and \eqref{quadgrad} simplify:
\begin{equation*}
|\s|^2=\lambda\xi-\xi^2,
\qquad
\nabla F=(\lambda-2\xi)\s.
\end{equation*}
Therefore by \eqref{divident} and \eqref{quaddiv}:
\begin{align*}
\Delta F
=-\Div \nabla F
&=(2\xi-\lambda)\Div\s+4|\s|^2  \\
&=-r\lambda^2+(n+5+2r)\lambda\xi-2(n+3)\xi^2.
\end{align*}
Furthermore, from \eqref{quad2cov} the spinnaker of $\s$ is:
$$
\z=(\lambda-2\xi)^2.
$$
Harmonicity equation \eqref{eigenharm} is therefore polynomial in $\xi$.  From the quartic terms:
$$
(n+3-2p)q=0.
$$
If $q=0$ then \eqref{eigenharm} reduces to:
$$
n+3+(2-r)p\lambda^2+(n+3+2(r-3)p)\lambda\xi-(n+3+(n-5)p)\xi^2=0,
$$
from the constant and linear terms of which:
$$
(r-2)p\lambda^2=n+3=2(3-r)p,
$$
forcing a disparity of sign.  Therefore $q\neq0$ and $2p=n+3$.  From the cubic terms:
$$
n+1-2r=0.
$$
Finally, the constant and linear/quadratic terms, respectively, yield:
\begin{gather}
\bigl((r+1)(r-2)+rq\bigr)\lambda^2=2(r+1), 
\label{quadconst} \\
\bigl(4(r+1)+\lambda^2\bigr)q=2(r+1)(2-r).
\label{quadlin}
\end{gather}
Eliminating $q$ yields equation \eqref{quadmu} for $\lambda_0$, and \eqref{quadlin} rearranges to \eqref{ddag}.
\end{proof}

The metric parameter $q$ of Theorem \ref{thmquad} is manifestly negative, with:
\begin{equation}
2q>2-r.
\label{quadloose}
\end{equation}  
The metric parameters $(p,q)$ for $\s$ therefore satisfy:
$$
q>\frac{1-r}{2}=1-\frac{p}{2}.
$$
Proposition \ref{propharm} then implies that:
$$
\tfrac14\/\lambda_0^{\,2}=\|\s\|_\infty^{\,2}>\frac{1}{p-1}=\frac1r.
$$
This allows the following modification of \eqref{quadmubounds}:
\begin{equation}
\frac{4}{r}<\lambda_0^{\,2}<\frac{4}{r-2},
\label{quadnewbounds}
\end{equation}
which is in fact a more accurate estimate when $r>8$.  Combining \eqref{quadloose} with the upper bound for $\lambda_0$ yields:
$$
q\/\|\s\|_\infty^{\,2}>\frac{2-r}{2(r-2)}=-\frac12.
$$
Therefore in all cases $\s$ is comfortably $q$-Riemannian.  Furthermore, from \eqref{quadconst} and \eqref{quadlin}:
\begin{equation*}
4rq=(r-2)\lambda_0^{\,2}-2(r-1)^2,
\end{equation*}
which when combined with \eqref{quadmubounds} yields tighter bounds for $q$:
\begin{equation}
\frac{3-2(r-1)^2}{4r}<q<\frac{4-2(r-1)^2}{4r},
\label{quadboundsq}
\end{equation}
provided $r\geqs4$, the upper bound being valid if $r\geqs3$.  (Use of \eqref{quadnewbounds} would marginally improve the lower bound when $r>8$.) 
For comparison, the precise values of $\lambda_0$ and $q$ for low dimensions may be tabulated:

\bigskip
\medskip
\centerline{
\vbox{\offinterlineskip
\hrule
\halign{
&\vrule#
&\strut
\hfill\quad\;\;$#$\quad\;\;\hfill\cr
height0pt&\omit&&\omit&&\omit&&\omit&&\omit&\cr
&n&&r&&p&&q&&\lambda_0^{\,2}/4&\cr
height0pt&\omit&&\omit&&\omit&&\omit&&\omit&\cr
\noalign{\hrule}
height2pt&\omit&&\omit&&\omit&&\omit&&\omit&\cr
&5&&3&&4&&\dfrac{1}{\sqrt3}-1&&\sqrt3-1&\cr
height4pt&\omit&&\omit&&\omit&&\omit&&\omit&\cr
\noalign{\hrule}
height4pt&\omit&&\omit&&\omit&&\omit&&\omit&\cr
&7&&4&&5&&\dfrac{\sqrt{201}-29}{16}&&\dfrac{\sqrt{201}-11}{8}&\cr
height4pt&\omit&&\omit&&\omit&&\omit&&\omit&\cr
\noalign{\hrule}
height4pt&\omit&&\omit&&\omit&&\omit&&\omit&\cr
&9&&5&&6&&\dfrac{\sqrt{34}-13}{5}&&\dfrac{\sqrt{34}-5}{3}&\cr
height4pt&\omit&&\omit&&\omit&&\omit&&\omit&\cr}
\hrule}}

\end{document}